\documentclass[10pt,oneside,a4paper]{article}
\usepackage[english]{babel}
\usepackage{amssymb}
\usepackage{amsmath}
\usepackage{amsthm}
\usepackage[mathscr]{eucal}
\usepackage{graphicx}
\usepackage{lmodern}
\usepackage[T1]{fontenc}
\usepackage{setspace}
\usepackage[latin1]{inputenc}
\usepackage{enumerate}
\usepackage{array}
\usepackage{bbm}
\usepackage{color}
\usepackage{sectsty}
    \sectionfont{\sc\Large}
\usepackage{titlesec}
\usepackage{indentfirst}
\titlelabel{\large{\thetitle}.\quad}
\newtheorem{theorem}{Theorem}[section]
\newtheorem{proposition}[theorem]{Proposition}
\newtheorem{lemma}[theorem]{Lemma}
\newtheorem{corollary}[theorem]{Corollary}
\newtheorem{definition}[theorem]{Definition}
\newtheorem{rmk}[theorem]{Remark}

\allowdisplaybreaks

\title{\textbf{Stochastic differential games involving impulse controls and double-obstacle quasi-variational inequalities}
\footnotetext{$^{*}$Tel.: +39-02-2399-4508; fax: +39-02-2399-4513.}
\footnotetext{\,\,\emph{E-mail address:} andrea.cosso@mail.polimi.it}}
\author{Andrea Cosso$^{*}$\\
\small Dipartimento di Matematica\\
\small Politecnico di Milano\\
\small piazza Leonardo da Vinci 32\\
\small 20133 Milano, Italy}
\date{}

\begin{document}
\maketitle

\begin{abstract}
We study a two-player zero-sum stochastic differential game with both players adopting impulse controls, on a finite time horizon. The Hamilton-Jacobi-Bellman-Isaacs (HJBI) partial differential equation of the game turns out to be a double-obstacle quasi-variational inequality, therefore the two obstacles are implicitly given. We prove that the upper and lower value functions coincide, indeed we show, by means of the dynamic programming principle for the stochastic differential game, that they are the unique viscosity solution to the HJBI equation, therefore proving that the game admits a value.
\end{abstract}

\section{Introduction}
The theory of two-player zero-sum differential games was pioneered by Isaacs \cite{I65}. Thanks to the mathematically rigorous definitions of upper and lower value functions for a differential game, see Elliott and Kalton~\cite{EK72}, Evans and Souganidis \cite{ES84} began to study differential games by means of the viscosity theory, characterizing the two value functions as the unique viscosity solutions to the corresponding Hamilton-Jacobi-Bellman-Isaacs partial differential equations (HJBI PDEs, for short).

Inspired by the results achieved in the deterministic case, Fleming and Souganidis~\cite{FS89}, for the first time, studied two-player zero-sum stochastic differential games using the viscosity theory. Adopting concepts presented in \cite{EK72}, they proved the dynamic programming principle for the stochastic differential game and showed that the upper and lower value functions are the unique viscosity solutions to the second order HJBI partial differential equations, which coincide under the Isaacs condition. \cite{FS89} is now considered a reference work in the field of stochastic differential games. More recently, the theory of backward stochastic differential equations has been successfully applied to the study of stochastic differential games, firstly by Hamad\`ene and Lepeltier \cite{HL95} and Hamad\`ene et al. \cite{HLP97}. Subsequently, Buckdahn and Li \cite{BL08} extended the findings presented in \cite{HL95}, \cite{HLP97}, and generalized the framework introduced in~\cite{FS89}.

The references mentioned above are focused on stochastic differential games with continuous controls. There exists also a large literature regarding Dynkin games, which are generalization of optimal stopping problems. Besides, the case with switching controls has been recently addressed by Tang and Hou~\cite{TH07}. In the present paper we study a two-player zero-sum stochastic differential game with both players adopting impulse controls, on a finite time horizon. In particular, we prove, using the dynamic programming principle, that the upper and lower value functions are the unique viscosity solution to the corresponding HJBI equation. To the author's knowledge, this kind of stochastic differential game has not yet been analyzed by means of the viscosity theory. As a matter of fact, only Zhang~\cite{Z11} studied, in the viscosity sense, a stochastic differential game involving impulse controls, but in \cite{Z11} one player adopts impulse controls, while the second player uses continuous controls.

Impulse control is a relevant topic in the field of stochastic control and a general reference is Bensoussan and Lions \cite{BL82}, where the focus is mainly on functional analysis methods. Instead, direct probabilistic methods are exploited in, for instance, Robin~\cite{R78} and Stettner \cite{S82}. Finally,  concerning the study of impulse control problems through the viscosity theory see, for example, Lenhart \cite{L89}, Tang and Yong \cite{TY93} and Kharroubi et al.~\cite{KMPZ10}.

During the past few years, the increasing demand for more realistic models in mathematical finance led to a renewed interest in impulse control. Indeed, impulse control may be particularly useful when dealing with, for instance, transaction costs and liquidity risk in financial markets. For more information on this subject see, for example, Korn~\cite{K99}, Ly Vath et al.~\cite{LMP07} and Bruder and Pham~\cite{BP09}.

We give an outline of the problem. Let $T>0$ be the finite time horizon of the game, $t\in[0,T]$ the initial time and $x\in\mathbb{R}^{n}$ the initial state. Then, the evolution of the state of the game is described by the following stochastic equation:
\begin{align*}
X_{s}=x+\int_{t}^{s}b(X_{r})\textup{d}r+\int_{t}^{s}\sigma(X_{r})\textup{d}W_{r}&
+\sum_{m\geqslant1}\xi_{m}\mathbbm{1}_{[\tau_{m},T]}(s)\prod_{\ell\geqslant1}\mathbbm{1}_{\{\tau_{m}\neq\rho_{\ell}\}}+ \\
&+\sum_{\ell\geqslant1}\eta_{\ell}\mathbbm{1}_{[\rho_{\ell},T]}(s),
\end{align*}
for all $s\in[t,T]$, $\mathbb{P}$-a.s., with $X_{t^{-}}=x$. Here $W$ is a $d$-dimensional Wiener process, while
\[
u(s)=\sum_{m\geqslant1}\xi_{m}\mathbbm{1}_{[\tau_{m},T]}(s) \qquad \text{ and } \qquad v(s)=\sum_{\ell\geqslant1}\eta_{\ell}\mathbbm{1}_{[\rho_{\ell},T]}(s)
\]
are the impulse controls of player I and player II, respectively. The random variables $\xi_{m}$ and $\eta_{\ell}$ take values in two convex cones $\mathscr{U}$ and $\mathscr{V}$ of $\mathbb{R}^{n}$, respectively, called the spaces of control actions. The infinite product $\prod_{\ell\geqslant1}\mathbbm{1}_{\{\tau_{m}\neq\rho_{\ell}\}}$ has the following meaning: When the two players act together on the system at the same time, then we take into account only the action of player~II. We denote by $X^{t,x;u,v}=\{X_{s}^{t,x;u,v},$ $t\leqslant s\leqslant T\}$ the state trajectory of the game with initial time $t$, initial state $x$ and impulse controls $u$ and $v$.

The gain functional for player I (resp., cost functional for player II) of the stochastic differential game is given by:
\begin{align*}
J(t,x;u,v)=\mathbb{E}\Big[\int_{t}^{T}f(X_{s}^{t,x;u,v})\textup{d}s&
-\sum_{m\geqslant1}c(\tau_{m},\xi_{m})\mathbbm{1}_{\{\tau_{m}\leqslant T\}}\prod_{\ell\geqslant1}\mathbbm{1}_{\{\tau_{m}\neq\rho_{\ell}\}}+ \\
&+\sum_{\ell\geqslant1}\chi(\rho_{\ell},\eta_{\ell})\mathbbm{1}_{\{\rho_{\ell}\leqslant T\}}+g(X_{T}^{t,x;u,v})\Big].
\end{align*}
Whenever player I performs an action he/she pays a positive cost $c$, which results in a gain for player II. Analogously, $\chi$ is the positive cost paid by player II to perform an impulse, which is a gain for player I. It is reasonable to impose the following condition on the cost function $\chi$:
\[
\chi(t,z_{1}+z_{2})\leqslant\chi(t,z_{1})+\chi(t,z_{2})-h(t),
\]
for every $t\in[0,T]$ and $z_{1},z_{2}\in\mathscr{V}$. The presence of a strictly positive function $h$ is required in the uniqueness proof for the HJBI equation. To guarantee uniqueness it is not enough to require the same condition on the other cost function $c$. We need to impose a stronger constraint that involves both cost functions, which is given by:
\[
c(t,y_{1}+z+y_{2})\leqslant c(t,y_{1})-\chi(t,z)+c(t,y_{2})-h(t),
\]
for every $y_{1},y_{2}\in\mathscr{U}$ and $z\in\mathscr{V}$. As a consequence, we have to require $\mathscr{V}\subset\mathscr{U}$. Finally, to prove the regularity with respect to time of the upper and lower value functions, we make the following assumption, introduced by Yong \cite{Y94} (see also Tang and Yong~\cite{TY93}), on the cost functions:
\[
c(t,y)\geqslant c(\hat{t},y) \qquad \text{ and } \qquad \chi(t,z)\geqslant\chi(\hat{t},z),
\]
for every $0\leqslant t\leqslant\hat{t}\leqslant T$, $y\in\mathscr{U}$ and $z\in\mathscr{V}$.

We define the upper and lower value functions using the Elliott-Kalton strategies as in \cite{FS89}. We show that they are the unique viscosity solution of a HJBI partial differential equation, which results to be the same for the two value functions. This is a consequence of the assumption that the two players can not act simultaneously on the system. Therefore the uniqueness for the HJBI equation implies that the upper and lower value functions coincide and the game admits a value.

The HJBI equation turns out to be a double-obstacle quasi-variational inequality, therefore the two obstacles are implicitly given and depend on the solution. This is a natural generalization of the result obtained in the context of Dynkin games, in which the HJBI equation is given by a double-obstacle variational inequality, see, for instance, Cvitani\'c and Karatzas \cite{CK96} and Hamad\`ene and Hassani \cite{HH05}. Single-obstacle quasi-variational inequality has been considered recently by Kharroubi et al. \cite{KMPZ10}, by means of backward stochastic differential equations with constrained jumps. In the following paper, instead, we use the dynamic programming principle for the stochastic differential game to prove that the two value functions are viscosity solutions to the HJBI equation.

The paper is organized as follows: Section~\ref{S:Problem} is devoted to fix the notations and to introduce rigorously the stochastic differential game. In Section~\ref{S:Regularity} we study the regularity properties of the upper and lower value functions. In particular, we show that they are continuous on $[0,T)\times\mathbb{R}^{n}$ and bounded. In Section~\ref{S:DPP} we prove the dynamic programming principle for the stochastic differential game. Furthermore, we deduce some corollaries and generalizations, which turn out to be useful to prove that the two value functions are viscosity solutions to the HJBI equation. This latter is the subject of Section~\ref{S:HJBI}, where we give also the definition of viscosity solution for the HJBI equation via test functions and by means of jets,
needed later in the proof of the uniqueness. Finally, in Section~\ref{S:Uniqueness} we prove the Comparison Theorem for the HJBI equation, therefore proving that the game admits a value.

\section{The Stochastic Differential Game}
\label{S:Problem}
In the present section we introduce the model associated to a two-player zero-sum stochastic differential game (SDG, for short) involving impulse controls with finite horizon.

Consider a complete probability space $(\Omega,\mathcal{F},\mathbb{P})$ and a $d$-dimensional standard Wiener process $W=(W_{t})_{t\geqslant0}$ defined on it. Let $\mathbb{F}=(\mathcal{F}_{t})_{t\geqslant0}$ be the natural filtration generated by the Wiener process, completed with the $\mathbb{P}$-null sets of $\mathcal{F}$. We are given two convex cones $\mathscr{U}$ and $\mathscr{V}$ of $\mathbb{R}^{n}$, with $\mathscr{V}\subset\mathscr{U}$. We call $\mathscr{U}$ and $\mathscr{V}$ the \emph{spaces of control actions}. We begin by introducing the concept of \emph{impulse control}.
\begin{definition}
An \textup{impulse control} $u=\sum_{m\geqslant1}\xi_{m}\mathbbm{1}_{[\tau_{m},T]}$ for player \textup{I} \textup{(}resp., $v=\sum_{\ell\geqslant1}\eta_{\ell}\mathbbm{1}_{[\rho_{\ell},T]}$ for player \textup{II}\textup{)} on $[t,T]\subset\mathbb{R}^{+}=[0,+\infty)$, is such that:
\begin{enumerate}[\upshape (i)]
  \item $(\tau_{m})_{m}$ \textup{(}resp., $(\rho_{\ell})_{\ell}$\textup{)}, the \textup{action times}, is a nondecreasing sequence of $\mathbb{F}$-stopping
        times, valued in $[t,T]\cup\{+\infty\}$.
  \item $(\xi_{m})_{m}$ \textup{(}resp., $(\eta_{\ell})_{\ell}$\textup{)}, the \textup{actions}, is a sequence of $\mathscr{U}$-valued \textup{(}resp., $\mathscr{V}$-valued\textup{)}
        random variables, where each $\xi_{m}$ \textup{(}resp., $\eta_{\ell}$\textup{)} is $\mathcal{F}_{\tau_{m}}$-measurable \textup{(}resp., $\mathcal{F}_{\rho_{\ell}}$-measurable\textup{)}.
\end{enumerate}
\end{definition}

We denote by $T>0$ the horizon of the stochastic differential game. Let $t\in[0,T]$ be the initial time of the game and $x\in\mathbb{R}^{n}$ the initial state. Then, given the impulse controls $u$ and $v$ on $[t,T]$, the state process of the stochastic differential game is defined as the solution to the following stochastic equation:
\begin{align}
\label{E:X}
X_{s}=x+\int_{t}^{s}b(X_{r})\textup{d}r+\int_{t}^{s}\sigma(X_{r})\textup{d}W_{r}&+\sum_{m\geqslant1}\xi_{m}\mathbbm{1}_{[\tau_{m},T]}(s)\prod_{\ell\geqslant1}\mathbbm{1}_{\{\tau_{m}\neq\rho_{\ell}\}}+ \notag \\
&+\sum_{\ell\geqslant1}\eta_{\ell}\mathbbm{1}_{[\rho_{\ell},T]}(s),
\end{align}
for all $s\in[t,T]$, $\mathbb{P}$-a.s., with $X_{t^{-}}=x$.
\begin{rmk}
\textup{As mentioned before, the infinite product $\prod_{\ell\geqslant1}\mathbbm{1}_{\{\tau_{m}\neq\rho_{\ell}\}}$ in equation (\ref{E:X}) means that when the two players act together on the system, then we take into account only the action of player \textup{II}.}
\end{rmk}

We make the following assumption on the functions $b$ and $\sigma$:
\begin{enumerate}
  \item[\textbf{(H$_{b,\sigma}$)}] The functions $b\colon\mathbb{R}^{n}\rightarrow\mathbb{R}^{n}$ and $\sigma\colon\mathbb{R}^{n}\rightarrow\mathbb{R}^{n\times d}$ are Lipschitz continuous and bounded.
\end{enumerate}

Thanks to assumption \textup{\textbf{(H$_{b,\sigma}$)}}, there exists a unique solution $X^{t,x;u,v}=\{X_{s}^{t,x;u,v},$ $t\leqslant s\leqslant T\}$ to equation (\ref{E:X}), for every $(t,x)\in[0,T]\times\mathbb{R}^{n}$, $u$ and $v$ impulse controls on $[t,T]$.
\begin{definition}
\label{D:Restriction}
Let $u=\sum_{m\geqslant1}\xi_{m}\mathbbm{1}_{[\tau_{m},T]}$ be an impulse control on $[t,T]$ and $\tau\leqslant\sigma$ two $[t,T]$-valued $\mathbb{F}$-stopping times. Then we define the \textup{restriction} $u_{[\tau,\sigma]}$ of the impulse control $u$ by:
\[
u_{[\tau,\sigma]}(s)=\sum_{m\geqslant1}\xi_{\mu_{t,\tau}(u)+m}\mathbbm{1}_{\{\tau_{\mu_{t,\tau}(u)+m}\leqslant s\leqslant\sigma\}}(s), \qquad \tau\leqslant s\leqslant\sigma,
\]
where $\mu_{t,\tau}(u)$ is the number of impulses up to time $\tau$, i.e.,
\begin{equation}
\label{E:mu}
\mu_{t,\tau}(u):=\sum_{m\geqslant1}\mathbbm{1}_{\{\tau_{m}\leqslant\tau\}}.
\end{equation}
\end{definition}

For every $(t,x)\in[0,T]\times\mathbb{R}^{n}$, $u$ and $v$ impulse controls on $[t,T]$, we define the \emph{gain functional} for player I (resp., cost functional for player II) as follows:
\begin{align}
\label{E:J}
J(t,x;u,v):=\mathbb{E}\Big[\int_{t}^{T}f(X_{s}^{t,x;u,v})\textup{d}s&
-\sum_{m\geqslant1}c(\tau_{m},\xi_{m})\mathbbm{1}_{\{\tau_{m}\leqslant T\}}\prod_{\ell\geqslant1}\mathbbm{1}_{\{\tau_{m}\neq\rho_{\ell}\}}+ \notag \\
&+\sum_{\ell\geqslant1}\chi(\rho_{\ell},\eta_{\ell})\mathbbm{1}_{\{\rho_{\ell}\leqslant T\}}+g(X_{T}^{t,x;u,v})\Big].
\end{align}
\begin{rmk}
\textup{As mentioned in the introduction, the function $c$ in (\ref{E:J}) is the cost function for player \textup{I} and is a gain function for player \textup{II}, meaning that when player \textup{I} performs an action has to pay a cost, resulting in a gain for player \textup{II}. Analogously, $\chi$ is the cost function for player \textup{II} and is a gain function for player \textup{I}.}
\end{rmk}

To guarantee a well defined gain functional we make the following assumptions on the functions $f$, $g$, $c$, $\chi$ and we introduce the concept of \emph{admissible impulse control}.
\begin{enumerate}
    \item[\textbf{(H$_{f,g}$)}] The \emph{running gain} $f\colon\mathbb{R}^{n}\rightarrow\mathbb{R}$ and the \emph{payoff} $g\colon\mathbb{R}^{n}\rightarrow\mathbb{R}$ are Lipschitz and bounded.
    \item[\textbf{(H$_{c,\chi}$})] The \emph{cost functions} $c\colon[0,T]\times\mathscr{U}\rightarrow\mathbb{R}$ and $\chi\colon[0,T]\times\mathscr{V}\rightarrow\mathbb{R}$ are $1/2$-H\"older continuous in time, uniformly with respect to the other variable. Furthermore
        \[
        \inf_{[0,T]\times\mathscr{U}}c>0, \qquad\qquad \inf_{[0,T]\times\mathscr{V}}\chi>0
        \]
        and there exists a function $h\colon[0,T]\rightarrow(0,\infty)$ such that for all $t\in[0,T]$
        \begin{equation}
        \label{E:c}
        c(t,y_{1}+z+y_{2})\leqslant c(t,y_{1})-\chi(t,z)+c(t,y_{2})-h(t)
        \end{equation}
        and
        \begin{equation}
        \label{E:chi}
        \chi(t,z_{1}+z_{2})\leqslant\chi(t,z_{1})+\chi(t,z_{2})-h(t),
        \end{equation}
        for every $y_{1},y_{2}\in\mathscr{U}$ and $z,z_{1},z_{2}\in\mathscr{V}$. Moreover
        \begin{equation}
        \label{E:c_chi_time}
        c(t,y)\geqslant c(\hat{t},y) \qquad \text{ and } \qquad \chi(t,z)\geqslant\chi(\hat{t},z),
        \end{equation}
        for every $t,\hat{t}\in[0,T]$, with $t\leqslant\hat{t}$, $y\in\mathscr{U}$ and $z\in\mathscr{V}$.
\end{enumerate}
\begin{definition}
An \textup{admissible impulse control} $u$ for player \textup{I} \textup{(}resp., $v$ for player \textup{II}\textup{)} on $[t,T]\subset\mathbb{R}^{+}$, is an impulse control for player \textup{I} \textup{(}resp., \textup{II}\textup{)} on $[t,T]$ with a finite average number of impulses, i.e.,
\[
\mathbb{E}[\mu_{t,T}(u)]<\infty \qquad \qquad \big(\text{resp., }\mathbb{E}[\mu_{t,T}(v)]<\infty\big),
\]
where $\mu_{t,T}(u)$ is given by equation (\ref{E:mu}). The set of all admissible impulse controls for player \textup{I} \textup{(}resp., \textup{II}\textup{)} on $[t,T]$ is denoted by $\mathcal{U}_{t,T}$ \textup{(}resp., $\mathcal{V}_{t,T}$\textup{)}. We identify two impulse controls $u=\sum_{m\geqslant1}\xi_{m}\mathbbm{1}_{[\tau_{m},T]}$ and $\bar{u}=\sum_{m\geqslant1}\bar{\xi}_{m}\mathbbm{1}_{[\bar{\tau}_{m},T]}$ in $\mathcal{U}_{t,T}$, and we write $u\equiv\bar{u}$ on $[t,T]$, if $\mathbb{P}(\{u=\bar{u}\text{ a.e. }on\text{ }[t,T]\})=1$. Similarly we interpret $v\equiv\bar{v}$ on $[t,T]$ in $\mathcal{V}_{t,T}$.
\end{definition}
\begin{rmk}
\textup{Under assumptions \textup{\textbf{(H$_{b,\sigma}$)}}, \textup{\textbf{(H$_{f,g}$)}} and \textup{\textbf{(H$_{c,\chi}$)}}, the gain functional $J(t,x;u,v)$, given by equation (\ref{E:J}), is well defined, for every $(t,x)\in[0,T]\times\mathbb{R}^{n}$, $u\in\mathcal{U}_{t,T}$ and $v\in\mathcal{V}_{t,T}$.}
\end{rmk}

We may now define the \emph{lower value function} $V^{-}$ and the \emph{upper value function} $V^{+}$ of the stochastic differential game. Before we need to introduce the concept of \emph{nonanticipative strategy}.
\begin{definition}
\label{D:Strategy}
A \textup{nonanticipative strategy} for player \textup{I} on $[t,T]\subset\mathbb{R}^{+}$ is a mapping
\[
\alpha\colon\mathcal{V}_{t,T}\longrightarrow\mathcal{U}_{t,T}
\]
such that for any stopping time $\tau\colon\Omega\longrightarrow[t,T]$ and any $v_{1}$, $v_{2}\in\mathcal{V}_{t,T}$, with $v_{1}\equiv v_{2}$ on $[\![t,\tau]\!]$, it holds that $\alpha(v_{1})\equiv\alpha(v_{2})$ on $[\![t,\tau]\!]$. Nonanticipative strategies for player \textup{II} on $[t,T]$, denoted by
\[
\beta\colon\mathcal{U}_{t,T}\longrightarrow\mathcal{V}_{t,T},
\]
are defined similarly. The set of all nonanticipative strategies $\alpha$ \textup{(}resp., $\beta$\textup{)} for player \textup{I} \textup{(}resp., \textup{II}\textup{)} on $[t,T]$ is denoted by $\mathcal{A}_{t,T}$ \textup{(}resp., $\mathcal{B}_{t,T}$\textup{)}.
\end{definition}

Hence, the two value functions are given by:
\begin{equation}
\label{E:V-}
V^{-}(t,x):=\inf_{\beta\in\mathcal{B}_{t,T}}\sup_{u\in\mathcal{U}_{t,T}}J(t,x;u,\beta(u))
\end{equation}
and
\begin{equation}
\label{E:V+}
V^{+}(t,x):=\sup_{\alpha\in\mathcal{A}_{t,T}}\inf_{v\in\mathcal{V}_{t,T}}J(t,x;\alpha(v),v)
\end{equation}
for every $(t,x)\in[0,T]\times\mathbb{R}^{n}$. If $V^{-}=V^{+}$ we say that the game admits a value and $V:=V^{-}=V^{+}$ is called the \emph{value function} of the game.

The \emph{Hamilton-Jacobi-Bellman-Isaacs equation} associated to the stochastic differential game, which turns out to be the same for the two value functions, because of the two players can not act simultaneously on the system, is given by:
\begin{equation}
\label{E:HJBI}
\begin{cases}
\vspace{.2cm}
\max\Big\{\min\Big[-\dfrac{\partial V}{\partial t}-\mathcal{L}V-f,V-\mathcal{H}_{\sup}^{c}V\Big],V-\mathcal{H}_{\inf}^{\chi}V\Big\}=0, \qquad &[0,T)\times\mathbb{R}^{n}, \\
V(T,x)=g(x), &\forall\,x\in\mathbb{R}^{n},
\end{cases}
\end{equation}
where $\mathcal{L}$ is the second-order local operator
\[
\mathcal{L}V=\langle b,\nabla_{x}V\rangle+\tfrac{1}{2}\textup{tr}\big[\sigma\sigma'\nabla_{x}^{2}V\big]
\]
and the nonlocal operators $\mathcal{H}_{\sup}^{c}$ and $\mathcal{H}_{\inf}^{\chi}$ are given by
\[
\mathcal{H}_{\sup}^{c}V(t,x)=\sup_{y\in\mathscr{U}}\big[V(t,x+y)-c(t,y)\big], \qquad
\mathcal{H}_{\inf}^{\chi}V(t,x)=\inf_{z\in\mathscr{V}}\big[V(t,x+z)+\chi(t,z)\big],
\]
for every $(t,x)\in[0,T]\times\mathbb{R}^{n}$.

\section{Regularity of the Value Functions}
\label{S:Regularity}
We prove that both the lower value function and the upper value function are Lipschitz with respect to the state variable, uniformly in time. Furthermore, they are H\"older continuous with exponent $1/2$ with respect to time on $[0,T)$, uniformly in the state variable. Finally we show that the two value functions are bounded.

We begin by proving the following lemma, which is concerned with the continuous dependence of $X^{t,x;u,v}$ with respect to $x$.

\begin{lemma}
\label{L:Cont-x}
Under assumption \textup{\textbf{(H$_{b,\sigma}$)}} there exists a constant $C>0$ such that for every $t\in[0,T]$, $x,\hat{x}\in\mathbb{R}^{n}$, $u\in\mathcal{U}_{t,T}$ and $v\in\mathcal{V}_{t,T}$ we have:
\begin{equation}
\mathbb{E}\big[|X_{s}^{t,x;u,v}-X_{s}^{t,\hat{x};u,v}|\big]\leqslant C|x-\hat{x}|, \notag
\end{equation}
for all $s\in[t,T]$.
\end{lemma}
\begin{proof}
We denote by $X:=X^{t,x;u,v}$ and $\hat{X}:=X^{t,\hat{x};u,v}$. Then an application of It\^o's formula gives
\begin{align*}
|X_{s}-\hat{X}_{s}|^{2}&=|x-\hat{x}|^{2}+2\int_{t}^{s}\langle X_{r}-\hat{X}_{r},b(X_{r})-b(\hat{X}_{r})\rangle\text{d}r+ \\
&+2\int_{t}^{s}\langle X_{r}-\hat{X}_{r},(\sigma(X_{r})-\sigma(\hat{X}_{r}))\text{d}W_{r}\rangle+\int_{t}^{s}|\sigma(X_{r})-\sigma(\hat{X}_{r})|^{2}\text{d}r.
\end{align*}
Therefore, thanks to assumption \textbf{(H$_{b,\sigma}$)}, there exists a constant $L>0$ such that
\[
\mathbb{E}\big[|X_{s}-\hat{X}_{s}|^{2}\big]\leqslant|x-\hat{x}|^{2}+L\int_{t}^{s}\mathbb{E}\big[|X_{r}-\hat{X}_{r}|^{2}\big]\text{d}r.
\]
From Gronwall's lemma we deduce the thesis.
\end{proof}

In the following proposition, using Lemma \ref{L:Cont-x}, we prove that the lower and upper value functions are Lipschitz with respect to the state variable, together with the gain functional.
\begin{proposition}
\label{P:Cont-x}
Under assumptions \textup{\textbf{(H$_{b,\sigma}$)}}, \textup{\textbf{(H$_{f,g}$)}} and \textup{\textbf{(H$_{c,\chi}$)}}, there exists a constant $C>0$ such that for every $t\in[0,T]$, $x,\hat{x}\in\mathbb{R}^{n}$, $u\in\mathcal{U}_{t,T}$ and $v\in\mathcal{V}_{t,T}$ we have:
\begin{align}
|J(t,x;u,v)-J(t,\hat{x};u,v)|+|V^{-}(t,x)&-V^{-}(t,\hat{x})|+|V^{+}(t,x)-V^{+}(t,\hat{x})|\leqslant C|x-\hat{x}|. \notag
\end{align}
\end{proposition}
\begin{proof}
It is enough to show that the conclusion holds true for the gain functional~$J$. Let us denote by $X:=X^{t,x;u,v}$ and $\hat{X}:=X^{t,\hat{x};u,v}$, then
\[
|J(t,x;u,v)-J(t,\hat{x};u,v)|\leqslant\mathbb{E}\Big[\int_{t}^{T}|f(X_{s})-f(\hat{X}_{s})|\textup{d}s+|g(X_{T})-g(\hat{X}_{T})|\Big].
\]
Thanks to assumption \textbf{(H$_{f,g}$)} there exists a constant $L>0$ such that
\[
|J(t,x;u,v)-J(t,\hat{x};u,v)|\leqslant L\int_{t}^{T}\mathbb{E}\big[|X_{s}-\hat{X}_{s}|\big]\textup{d}s+L\mathbb{E}\big[|X_{T}-\hat{X}_{T}|\big].
\]
From Lemma \ref{L:Cont-x} we get the thesis.
\end{proof}

Now we prove the regularity condition of the value functions with respect to time and we need the following lemma.
\begin{lemma}
\label{L:Cont-t}
Under assumptions \textup{\textbf{(H$_{b,\sigma}$)}}, \textup{\textbf{(H$_{f,g}$)}} and \textup{\textbf{(H$_{c,\chi}$)}} the lower and upper value functions are given by:
\[
V^{-}(t,x)=\inf_{\beta\in\bar{\mathcal{B}}_{t,T}}\sup_{u\in\bar{\mathcal{U}}_{t,T}}J(t,x;u,\beta(u))
\]
and
\[
V^{+}(t,x)=\sup_{\alpha\in\bar{\mathcal{A}}_{t,T}}\inf_{v\in\bar{\mathcal{V}}_{t,T}}J(t,x;\alpha(v),v),
\]
for every $(t,x)\in[0,T)\times\mathbb{R}^{n}$, where $\bar{\mathcal{U}}_{t,T}$ and $\bar{\mathcal{V}}_{t,T}$ contain all the impulse controls in $\mathcal{U}_{t,T}$ and $\mathcal{V}_{t,T}$, respectively, which have no impulses at time $t$. Similarly, $\bar{\mathcal{A}}_{t,T}$ and $\bar{\mathcal{B}}_{t,T}$ are subsets of $\mathcal{A}_{t,T}$ and $\mathcal{B}_{t,T}$, respectively. In particular, they contain all the nonanticipative strategies with values in $\bar{\mathcal{U}}_{t,T}$ and $\bar{\mathcal{V}}_{t,T}$, respectively.
\end{lemma}
\begin{proof}
Let $\varepsilon>0$, $u\in\mathcal{U}_{t,T}\backslash\bar{\mathcal{U}}_{t,T}$ and $\beta\in\mathcal{B}_{t,T}\backslash\bar{\mathcal{B}}_{t,T}$. We have to prove that there exist $\bar{u}\in\bar{\mathcal{U}}_{t,T}$ and $\bar{\beta}\in\bar{\mathcal{B}}_{t,T}$ such that
\[
|J(t,x;u,\beta(u))-J(t,x;\bar{u},\bar{\beta}(\bar{u}))|\leqslant\varepsilon.
\]
Let $v:=\beta(u)\in\mathcal{V}_{t,T}$ and $\bar{\beta}(\tilde{u})=\bar{v}\in\bar{\mathcal{V}}_{t,T}$, for every $\tilde{u}\in\mathcal{U}_{t,T}$. Then, given $u\in\mathcal{U}_{t,T}\backslash\bar{\mathcal{U}}_{t,T}$ and $v\in\mathcal{V}_{t,T}$, we have to prove that there exist $\bar{u}\in\bar{\mathcal{U}}_{t,T}$ and $\bar{v}\in\bar{\mathcal{V}}_{t,T}$ such that
\[
|J(t,x;u,v)-J(t,x;\bar{u},\bar{v})|\leqslant\varepsilon.
\]
We may suppose $v\in\mathcal{V}_{t,T}\backslash\bar{\mathcal{V}}_{t,T}$, in the other case the proof is simpler and we omit it.

We start by considering the case in which $u$ and $v$ have only a single impulse at time~$t$. As a consequence, there exist two $[t,T]$-valued $\mathbb{F}$-stopping times $\tau$ and $\rho$, with $\mathbb{P}(\tau=t)>0$ and $\mathbb{P}(\rho=t)>0$, such that
\[
u=\xi\mathbbm{1}_{[\tau,T]}+\hat{u} \qquad \text{ and } \qquad v=\eta\mathbbm{1}_{[\rho,T]}+\hat{v},
\]
where $\hat{u}=\sum_{m\geqslant1}\xi_{m}\mathbbm{1}_{[\tau_{m},T]}\in\bar{\mathcal{U}}_{[t,T]}$, $\hat{v}=\sum_{\ell\geqslant1}\eta_{\ell}\mathbbm{1}_{[\rho_{\ell},T]}\in\bar{\mathcal{V}}_{[t,T]}$, $\xi$ is an $\mathcal{F}_{\tau}$-measurable $\mathscr{U}$-valued random variable and $\eta$ is an $\mathcal{F}_{\rho}$-measurable $\mathscr{V}$-valued random variable.

For every integer $n\geqslant1/(T-t)$, we introduce the following $\mathbb{F}$-stopping times:
\[
\tau_{n}=\big(\tau+\tfrac{1}{n}\big)\mathbbm{1}_{\{\tau=t\}}+\tau\mathbbm{1}_{\{\tau>t\}} \qquad \text{ and } \qquad
\rho_{n}=\big(\rho+\tfrac{1}{n}\big)\mathbbm{1}_{\{\rho=t\}}+\rho\mathbbm{1}_{\{\rho>t\}}.
\]
Define the admissible impulse controls:
\[
u_{n}=\xi\mathbbm{1}_{[\tau_{n},T]}+\hat{u} \quad \text{ and } \quad v_{n}=\eta\mathbbm{1}_{[\rho_{n},T]}+\hat{v}.
\]
Note that $u_{n}\in\bar{\mathcal{U}}_{[t,T]}$ and $v_{n}\in\bar{\mathcal{V}}_{[t,T]}$. Then we have:
\begin{align*}
&J(t,x;u,v)-J(t,x;u_{n},v_{n})=\mathbb{E}\Big[\int_{t}^{T}\big(f(X_{s}^{t,x;u,v})-f(X_{s}^{t,x;u_{n},v_{n}})\big)\text{d}s \\
&-\sum_{m\geqslant1}c(\tau_{m},\xi_{m})\mathbbm{1}_{\{\tau_{m}\leqslant T\}}
\prod_{\ell\geqslant1}\mathbbm{1}_{\{\tau_{m}\neq\rho_{\ell}\}}\big(\mathbbm{1}_{\{\tau_{m}\neq\rho\}}-\mathbbm{1}_{\{\tau_{m}\neq\rho_{n}\}}\big) \\
&-c(\tau,\xi)\mathbbm{1}_{\{\tau\leqslant T\}}\prod_{\ell\geqslant1}\mathbbm{1}_{\{\tau\neq\rho_{\ell}\}}\mathbbm{1}_{\{\tau\neq\rho\}}
+c(\tau_{n},\xi)\mathbbm{1}_{\{\tau_{n}\leqslant T\}}\prod_{\ell\geqslant1}\mathbbm{1}_{\{\tau_{n}\neq\rho_{\ell}\}}\mathbbm{1}_{\{\tau_{n}\neq\rho_{n}\}}+ \\
&+\chi(\rho,\eta)\mathbbm{1}_{\{\rho\leqslant T\}}-\chi(\rho_{n},\eta)\mathbbm{1}_{\{\rho_{n}\leqslant T\}}
+g(X_{T}^{t,x;u,v})-g(X_{T}^{t,x;u_{n},v_{n}})\Big].
\end{align*}
Now we observe that $\tau_{n}\rightarrow\tau$ and $\rho_{n}\rightarrow\rho$, as $n$ tends to infinity, $\mathbb{P}$-a.s.. Moreover, $\forall s\in(t,T]$,
$X_{s}^{t,x;u_{n},v_{n}}\rightarrow X_{s}^{t,x;u,v}$, as $n$ tends to infinity, $\mathbb{P}$-a.s.. Therefore, from the Dominated Convergence Theorem, we deduce the existence of an integer $N\geqslant1$ such that
\[
|J(t,x;u,v)-J(t,x;u_{n},v_{n})|\leqslant\varepsilon, \qquad \forall n\geqslant N.
\]
Finally, we have to consider the case in which $u$ or $v$ or both have multiple impulses at time $t$. However it is simple to show, using conditions (\ref{E:c}) and (\ref{E:chi}), that we can reduce this case to the previous one with only a single impulse at time $t$.
\end{proof}
\begin{proposition}
\label{P:Cont-t}
Under assumptions \textup{\textbf{(H$_{b,\sigma}$)}}, \textup{\textbf{(H$_{f,g}$)}} and \textup{\textbf{(H$_{c,\chi}$)}}, there exists a constant $C>0$ such that:
\[
|V^{-}(t,x)-V^{-}(\hat{t},x)|+|V^{+}(t,x)-V^{+}(\hat{t},x)|\leqslant C|t-\hat{t}|^{\frac{1}{2}},
\]
for every $t,\hat{t}\in[0,T)$, $x\in\mathbb{R}^{n}$.
\end{proposition}
\begin{rmk}
\label{R:ContExt}
\textup{Observe that the restriction of the lower value function to $[0,T)\times\mathbb{R}^{n}$ admits an extension to $[0,T]\times\mathbb{R}^{n}$, $1/2$-H\"older continuous in time and Lipschitz continuous in the state variable, which in general does not coincide with the value function itself at the horizon time $T$. An analogous remark applies to the upper value function~$V^{+}$.}
\end{rmk}
\begin{proof}
We make the proof only for the lower value function $V^{-}$, the other case being analogous. Let $\hat{t}\in[t,T]$, then for every $\varepsilon>0$, thanks to Lemma~\ref{L:Cont-t}, there exist $u_{\varepsilon}\in\mathcal{U}_{t,T}$ and $\beta_{\varepsilon}\in\bar{\mathcal{B}}_{\hat{t},T}$ such that
\[
V^{-}(t,x)-V^{-}(\hat{t},x)\leqslant J(t,x;u_{\varepsilon},\hat{\beta}_{\varepsilon}(u_{\varepsilon}))-J(\hat{t},x;\hat{u}_{\varepsilon},\beta_{\varepsilon}(\hat{u}_{\varepsilon}))+\varepsilon,
\]
where $\hat{u}_{\varepsilon}\in\mathcal{U}_{\hat{t},T}$ and $\hat{\beta}_{\varepsilon}\in\mathcal{B}_{t,T}$ will be choosen later. In particular, suppose that $u_{\varepsilon}=\sum_{m\geqslant1}\xi_{m}^{\varepsilon}\mathbbm{1}_{[\tau_{m}^{\varepsilon},T]}\in\mathcal{U}_{t,T}$, then define $\hat{u}_{\varepsilon}$ as follows:
\[
\hat{u}_{\varepsilon}=\sum_{\tau_{m}^{\varepsilon}\leqslant\hat{t}}\xi_{m}^{\varepsilon}\mathbbm{1}_{\hat{t}}
+\sum_{\tau_{m}^{\varepsilon}>\hat{t}}\xi_{m}^{\varepsilon}\mathbbm{1}_{[\tau_{m}^{\varepsilon},T]}.
\]
Thus $\hat{u}_{\varepsilon}$ is nothing but the impulse control obtained from $u_{\varepsilon}$ by moving all the impulses in the time interval $[t,\hat{t}]$ to the instant $\hat{t}$. Now define $\hat{\beta}_{\varepsilon}(u)=\beta_{\varepsilon}(\hat{u}_{\varepsilon})=:v_{\varepsilon}\in\bar{\mathcal{V}}_{t,T}$, for every $u\in\mathcal{U}_{t,T}$. Then we have:
\begin{equation}
\label{P:Cont-t,Bounds0}
V^{-}(t,x)-V^{-}(\hat{t},x)\leqslant J(t,x;u_{\varepsilon},v_{\varepsilon})-J(\hat{t},x;\hat{u}_{\varepsilon},v_{\varepsilon})+\varepsilon.
\end{equation}
Using conditions (\ref{E:c}) and (\ref{E:c_chi_time}) we find
\[
\sum_{\tau_{m}^{\varepsilon}\leqslant\hat{t}}c(\tau_{m}^{\varepsilon},\xi_{m}^{\varepsilon})\geqslant c\Big(\hat{t},\sum_{\tau_{m}^{\varepsilon}\leqslant\hat{t}}\xi_{m}^{\varepsilon}\Big),
\]
if there is at least one impulse in the time interval $[t,\hat{t}]$.
Therefore
\begin{align*}
J(t,x;u_{\varepsilon},v_{\varepsilon})-J(\hat{t},x;\hat{u}_{\varepsilon},v_{\varepsilon})\leqslant\mathbb{E}\Big[&\int_{t}^{T}f(X_{s}^{t,x;u_{\varepsilon},v_{\varepsilon}})\text{d}s
-\int_{\hat{t}}^{T}f(X_{s}^{\hat{t},x;\hat{u}_{\varepsilon},v_{\varepsilon}})\text{d}s+ \\
&+g(X_{T}^{t,x;u_{\varepsilon},v_{\varepsilon}})-g(X_{T}^{\hat{t},x;\hat{u}_{\varepsilon},v_{\varepsilon}})\Big].
\end{align*}
We note that, thanks to assumption \text{\textbf{(H$_{b,\sigma}$)}}, there exists a constant $C_{1}>0$ such that
\[
\mathbb{E}\big[|X_{s}^{t,x;u_{\varepsilon},v_{\varepsilon}}-X_{s}^{\hat{t},x;\hat{u}_{\varepsilon},v_{\varepsilon}}|\big]\leqslant C_{1}|t-\hat{t}|^{\frac{1}{2}},
\]
for all $s\in[\hat{t},T]$. Therefore we find, using also the boundedness of $f$, that there exists a constant $C_{2}>0$ such that
\[
J(t,x;u_{\varepsilon},v_{\varepsilon})-J(\hat{t},x;\hat{u}_{\varepsilon},v_{\varepsilon})\leqslant C_{2}|t-\hat{t}|^{\frac{1}{2}}.
\]
Hence
\[
V^{-}(t,x)-V^{-}(\hat{t},x)\leqslant C_{2}|t-\hat{t}|^{\frac{1}{2}}.
\]
In a similar way we can prove that there exists a constant $C_{3}>0$ such that
\[
V^{-}(t,x)-V^{-}(\hat{t},x)\geqslant-C_{3}|t-\hat{t}|^{\frac{1}{2}},
\]
from which we deduce the thesis.
\end{proof}

Finally, in the following proposition, we prove that the two value functions are bounded.
\begin{proposition}
\label{P:VBounded}
Under assumptions \textup{\textbf{(H$_{b,\sigma}$)}}, \textup{\textbf{(H$_{f,g}$)}} and \textup{\textbf{(H$_{c,\chi}$)}}, the lower and upper value functions are bounded.
\end{proposition}
\begin{proof}
We make the proof for the lower value function, the other case being analogous. Let $\varepsilon>0$, then, using the definition of lower value function, equation (\ref{E:V-}), we have, for $(t,x)\in[0,T]\times\mathbb{R}^{n}$,
\begin{align*}
V^{-}(t,x)=&\,
\inf_{\beta\in\mathcal{B}_{t,T}}\sup_{u\in\mathcal{U}_{t,T}}\mathbb{E}\Big[\int_{t}^{T}f(X_{s}^{t,x;u,\beta(u)})\textup{d}s+g(X_{T}^{t,x;u,\beta(u)})+ \\
&\hspace{-1.5cm}+\sum_{\ell\geqslant1}\chi(\rho_{\ell}(u),\eta_{\ell}(u))\mathbbm{1}_{\{\rho_{\ell}(u)\leqslant T\}}
-\sum_{m\geqslant1}c(\tau_{m},\xi_{m})\mathbbm{1}_{\{\tau_{m}\leqslant T\}}\prod_{\ell\geqslant1}\mathbbm{1}_{\{\tau_{m}\neq\rho_{\ell}(u)\}}\Big]
\geqslant \\
&\hspace{-1.5cm}\geqslant
\mathbb{E}\Big[\int_{t}^{T}f(X_{s}^{t,x;u_{0},\beta_{\varepsilon}(u_{0})})\textup{d}s+g(X_{T}^{t,x;u_{0},\beta_{\varepsilon}(u_{0})})+\sum_{\ell\geqslant1}\chi(\rho_{\ell}^{\varepsilon}(u_{0}),\eta_{\ell}^{\varepsilon}(u_{0}))\mathbbm{1}_{\{\rho_{\ell}^{\varepsilon}(u_{0})\leqslant T\}}\Big] \\
&\hspace{-1.5cm}-\varepsilon\geqslant
\mathbb{E}\Big[\int_{t}^{T}f(X_{s}^{t,x;u_{0},\beta_{\varepsilon}(u_{0})})\textup{d}s+g(X_{T}^{t,x;u_{0},\beta_{\varepsilon}(u_{0})})\Big]-\varepsilon,
\end{align*}
for some $\beta_{\varepsilon}(u_{0})=\sum_{\ell\geqslant1}\eta_{\ell}^{\varepsilon}(u_{0})\mathbbm{1}_{[\rho_{\ell}^{\varepsilon}(u_{0}),T]}\in\mathcal{V}_{t,T}$, where $u_{0}\in\mathcal{U}_{t,T}$ is the control with no impulses. Since $f$ and $g$ are bounded we deduce that $V^{-}$ is bounded from below. In a similar way, we can prove that $V^{-}$ is also bounded from above.
\end{proof}

\section{Dynamic Programming Principle}
\label{S:DPP}
We now present the dynamic programming principle (DPP) for the stochastic differential game.
\begin{theorem}
\label{T:DPP}
Under assumptions \textup{\textbf{(H$_{b,\sigma}$)}}, \textup{\textbf{(H$_{f,g}$)}} and \textup{\textbf{(H$_{c,\chi}$)}}, given $0\leqslant t\leqslant s<T$ and $x\in\mathbb{R}^{n}$, we have:
\begin{align}
\label{E:DPP-}
V^{-}(t,x):=\inf_{\beta\in\mathcal{B}_{t,T}}\sup_{u\in\mathcal{U}_{t,T}}\mathbb{E}\Big[\int_{t}^{s}f(X_{r}^{t,x;u,\beta(u)})\textup{d}r&
+\sum_{\ell\geqslant1}\chi(\rho_{\ell}(u),\eta_{\ell}(u))\mathbbm{1}_{\{\rho_{\ell}(u)\leqslant s\}} \notag \\
&\hspace{-5cm}-\sum_{m\geqslant1}c(\tau_{m},\xi_{m})\mathbbm{1}_{\{\tau_{m}\leqslant s\}}\prod_{\ell\geqslant1}\mathbbm{1}_{\{\tau_{m}\neq\rho_{\ell}(u)\}}
+V^{-}(s,X_{s}^{t,x;u,\beta(u)})\Big]
\end{align}
and
\begin{align}
\label{E:DPP+}
V^{+}(t,x):=\sup_{\alpha\in\mathcal{A}_{t,T}}\inf_{v\in\mathcal{V}_{t,T}}\mathbb{E}\Big[\int_{t}^{s}f(X_{r}^{t,x;\alpha(v),v})\textup{d}r&
+\sum_{\ell\geqslant1}\chi(\rho_{\ell},\eta_{\ell})\mathbbm{1}_{\{\rho_{\ell}\leqslant s\}} \\
&\hspace{-5cm}-\sum_{m\geqslant1}c(\tau_{m}(v),\xi_{m}(v))\mathbbm{1}_{\{\tau_{m}(v)\leqslant s\}}\prod_{\ell\geqslant1}\mathbbm{1}_{\{\tau_{m}(v)\neq\rho_{\ell}\}}
+V^{+}(s,X_{s}^{t,x;\alpha(v),v})\Big]. \notag
\end{align}
\end{theorem}
\begin{proof}
We prove the dynamic programming principle only for $V^{-}$, the other case being analogous. Let $u\in\mathcal{U}_{t,T}$ and $\varepsilon>0$, then there exists a strategy $\beta^{1,\varepsilon}\in\mathcal{B}_{t,T}$ such that
\begin{align}
\label{E:DPPproof}
&\inf_{\beta\in\mathcal{B}_{t,T}}\sup_{u\in\mathcal{U}_{t,T}}\mathbb{E}\Big[\int_{t}^{s}f(X_{r}^{t,x;u,\beta(u)})\textup{d}r
-\sum_{m\geqslant1}c(\tau_{m},\xi_{m})\mathbbm{1}_{\{\tau_{m}\leqslant s\}}\prod_{\ell\geqslant1}\mathbbm{1}_{\{\tau_{m}\neq\rho_{\ell}(u)\}}+ \\
&\quad+\sum_{\ell\geqslant1}\chi(\rho_{\ell}(u),\eta_{\ell}(u))\mathbbm{1}_{\{\rho_{\ell}(u)\leqslant s\}}+V^{-}(s,X_{s}^{t,x;u,\beta(u)})\Big]
\geqslant\mathbb{E}\Big[\int_{t}^{s}f(X_{r}^{t,x;u,\beta^{1,\varepsilon}(u)})\textup{d}r \notag \\
&\quad-\sum_{m\geqslant1}c(\tau_{m},\xi_{m})\mathbbm{1}_{\{\tau_{m}\leqslant s\}}\prod_{\ell\geqslant1}\mathbbm{1}_{\{\tau_{m}\neq\rho_{\ell}^{1,\varepsilon}(u)\}}
+\sum_{\ell\geqslant1}\chi(\rho_{\ell}^{1,\varepsilon}(u),\eta_{\ell}^{1,\varepsilon}(u))\mathbbm{1}_{\{\rho_{\ell}^{1,\varepsilon}(u)\leqslant s\}}+ \notag \\
&\quad+V^{-}(s,X_{s}^{t,x;u,\beta^{1,\varepsilon}(u)})\Big]-\varepsilon. \notag
\end{align}
Now, from the regularity of $V^{-}$ and $J$ with respect to the state variable, see Proposition~\ref{P:Cont-x}, we deduce the existence of a strategy $\beta^{2,\varepsilon}\in\bar{\mathcal{B}}_{s,T}$ such that
\begin{equation}
\label{E:DPPproof2}
\mathbb{E}\big[V^{-}(s,X_{s}^{t,x;u,\beta^{1,\varepsilon}(u)})\big]\geqslant
\mathbb{E}\big[J\big(s,X_{s}^{t,x;u,\beta^{1,\varepsilon}(u)};u_{[s,T]},\beta^{2,\varepsilon}(u_{[s,T]})\big)\big]-\varepsilon,
\end{equation}
where $u_{[s,T]}$ is introduced in Definition \ref{D:Restriction}. Indeed, let $(A_{i})_{i\geqslant1}$ be a partition of $\mathbb{R}^{n}$ such that, thanks to the regularity of $V^{-}$ and $J$, given $x_{i}\in A_{i}$, then for all $y\in A_{i}$ we have
\[
J(s,x_{i};u_{[s,T]},v)\geqslant
J(s,y;u_{[s,T]},v)-\tfrac{1}{3}\varepsilon \qquad \text{ and } \qquad
V^{-}(s,y)\geqslant V^{-}(s,x_{i})-\tfrac{1}{3}\varepsilon,
\]
for every $v\in\bar{\mathcal{V}}_{s,T}$. Moreover, thanks to Lemma \ref{L:Cont-t}, for every $A_{i}$ there exists a strategy $\beta^{A_{i}}\in\bar{\mathcal{B}}_{s,T}$ such that
\[
V^{-}(s,x_{i})\geqslant J(s,x_{i};u_{[s,T]},\beta^{A_{i}}(u_{[s,T]}))-\tfrac{1}{3}\varepsilon.
\]
As a consequence, for every $y\in A_{i}$ the following inequality holds:
\[
V^{-}(s,y)\geqslant J(s,y;u_{[s,T]},\beta^{A_{i}}(u_{[s,T]}))-\varepsilon.
\]
Therefore, in our case, we have
\begin{align*}
&\mathbb{E}\Big[V^{-}(s,X_{s}^{t,x;u,\beta^{1,\varepsilon}(u)})\Big]=
\mathbb{E}\Big[\sum_{i\geqslant1}V^{-}(s,X_{s}^{t,x;u,\beta^{1,\varepsilon}(u)})\mathbbm{1}_{A_{i}}(X_{s}^{t,x;u,\beta^{1,\varepsilon}(u)})\Big]\geqslant \\
&\qquad\geqslant\mathbb{E}\Big[\sum_{i\geqslant1}J\big(s,X_{s}^{t,x;
u,\beta^{1,\varepsilon}(u)};u_{[s,T]},\beta^{A_{i}}(u_{[s,T]})\big)\mathbbm{1}_{A_{i}}(X_{s}^{t,x;u,\beta^{1,\varepsilon}(u_{[s,T]})})\Big]
-\varepsilon= \\
&\qquad=\mathbb{E}\Big[J\Big(s,X_{s}^{t,x;u,\beta^{1,\varepsilon}(u_{[s,T]})};
u_{[s,T]},\sum_{i=1}^{+\infty}\mathbbm{1}_{A_{i}}(X_{s}^{t,x;u,\beta^{1,\varepsilon}(u_{[s,T]})})\beta^{A_{i}}(u_{[s,T]})\Big)\Big]
-\varepsilon.
\end{align*}
We introduce the strategy $\beta^{2,\varepsilon}\in\bar{\mathcal{B}}_{s,T}$ given by
\[
\beta^{2,\varepsilon}(u):=\sum_{i=1}^{+\infty}\mathbbm{1}_{A_{i}}(X_{s}^{t,x;u,\beta^{1,\varepsilon}(u)})\beta^{A_{i}}(u_{[s,T]}).
\]
This means that if $\beta^{A_{i}}(u_{[s,T]})=\sum_{\ell\geqslant1}\eta_{\ell}^{A_{i}}(u_{[s,T]})\mathbbm{1}_{[\rho_{\ell}^{A_{i}}(u_{[s,T]}),T]}$, then $\beta^{2,\varepsilon}(u)=\sum_{\ell\geqslant1}$ $\eta_{\ell}^{2,\varepsilon}(u)\mathbbm{1}_{[\rho_{\ell}^{2,\varepsilon}(u),T]}$ is given by
\[
\rho_{\ell}^{2,\varepsilon}(u):=\sum_{i=1}^{+\infty}\mathbbm{1}_{A_{i}}(X_{s}^{t,x;u,\beta^{1,\varepsilon}(u)})\rho_{\ell}^{A_{i}}(u_{[s,T]})
\]
and
\[
\eta_{\ell}^{2,\varepsilon}(u):=\sum_{i=1}^{+\infty}\mathbbm{1}_{A^{i}}(X_{s}^{t,x;u,\beta^{1,\varepsilon}(u)})\eta_{\ell}^{A^{i}}(u_{[s,T]}).
\]
Finally, we define a new strategy $\beta^{\varepsilon}\in\mathcal{B}_{t,T}$, equal to $\beta^{1,\varepsilon}$ up to time $s$ and to $\beta^{2,\varepsilon}$ from time $s$ to $T$. More precisely, the definition of $\beta^{\varepsilon}$ is as follows:
Let $\beta^{1,\varepsilon}(u)=\sum_{\ell\geqslant1}\eta_{\ell}^{1,\varepsilon}(u)\mathbbm{1}_{[\rho_{\ell}^{1,\varepsilon}(u),T]}$,
$\beta^{2,\varepsilon}(u)=\sum_{\ell\geqslant1}\eta_{\ell}^{2,\varepsilon}(u)\mathbbm{1}_{[\rho_{\ell}^{2,\varepsilon}(u),T]}$ and
$\beta^{\varepsilon}(u)=\sum_{\ell\geqslant1}\eta_{\ell}^{\varepsilon}(u)$ $\mathbbm{1}_{[\rho_{\ell}^{\varepsilon}(u),T]}$, then
\[
\rho_{\ell}^{\varepsilon}(u):=\rho_{\ell}^{1,\varepsilon}(u)\mathbbm{1}_{\{\ell\leqslant\mu_{t,s}(\beta^{1,\varepsilon}(u))\}}
+\rho_{\ell-\mu_{t,s}(\beta^{1,\varepsilon}(u))}^{2,\varepsilon}(u)\mathbbm{1}_{\{\ell>\mu_{t,s}(\beta^{1,\varepsilon}(u))\}}
\]
and
\[
\eta_{\ell}^{\varepsilon}(u):=\eta_{\ell}^{1,\varepsilon}(u)\mathbbm{1}_{\{\ell\leqslant\mu_{t,s}(\beta^{1,\varepsilon}(u))\}}
+\eta_{\ell-\mu_{t,s}(\beta^{1,\varepsilon}(u))}^{2,\varepsilon}(u)\mathbbm{1}_{\{\ell>\mu_{t,s}(\beta^{1,\varepsilon}(u))\}},
\]
where $\mu_{t,s}(\beta^{1,\varepsilon}(u))$ is given by equation (\ref{E:mu}). Therefore, from (\ref{E:DPPproof}) and (\ref{E:DPPproof2}) we get
\begin{align*}
&\mathbb{E}\Big[\int_{t}^{s}f(X_{r}^{t,x;u,\beta^{1,\varepsilon}(u)})\textup{d}r
-\sum_{m\geqslant1}c(\tau_{m},\xi_{m})\mathbbm{1}_{\{\tau_{m}\leqslant s\}}\prod_{\ell\geqslant1}\mathbbm{1}_{\{\tau_{m}\neq\rho_{\ell}^{1,\varepsilon}(u)\}}+ \\
&\hspace{1.5cm}+\sum_{\ell\geqslant1}\chi(\rho_{\ell}^{1,\varepsilon}(u),\eta_{\ell}^{1,\varepsilon}(u))\mathbbm{1}_{\{\rho_{\ell}^{1,\varepsilon}(u)\leqslant s\}}+V^{-}(s,X_{s}^{t,x;u,\beta^{1,\varepsilon}(u)})\Big]-\varepsilon\geqslant \\
&\hspace{1.5cm}\geqslant\mathbb{E}\Big[\int_{t}^{s}f(X_{r}^{t,x;u,\beta^{1,\varepsilon}(u)})\textup{d}r-\sum_{m\geqslant1}c(\tau_{m},\xi_{m})\mathbbm{1}_{\{\tau_{m}\leqslant s\}}\prod_{\ell\geqslant1}\mathbbm{1}_{\{\tau_{m}\neq\rho_{\ell}^{1,\varepsilon}(u)\}}+ \\
&\hspace{1.5cm}+\sum_{\ell\geqslant1}\chi(\rho_{\ell}^{1,\varepsilon}(u),\eta_{\ell}^{1,\varepsilon}(u))\mathbbm{1}_{\{\rho_{\ell}^{1,\varepsilon}(u)\leqslant s\}}+J(s,X_{s}^{t,x;u,\beta^{1,\varepsilon}(u)};u_{[s,T]},\beta^{2,\varepsilon}(u))\Big] \\
&\hspace{1.5cm}-2\varepsilon=J(t,x;u,\beta^{\varepsilon}(u))-2\varepsilon.
\end{align*}
We can now easily deduce that the following inequality holds:
\begin{align*}
V^{-}(t,x)\leqslant\inf_{\beta\in\mathcal{B}_{t,T}}\sup_{u\in\mathcal{U}_{t,T}}\mathbb{E}\Big[\int_{t}^{s}f(X_{r}^{t,x;u,\beta(u)})\textup{d}r&
+\sum_{\ell\geqslant1}\chi(\rho_{\ell}(u),\eta_{\ell}(u))\mathbbm{1}_{\{\rho_{\ell}(u)\leqslant s\}} \\
&\hspace{-4cm}-\sum_{m\geqslant1}c(\tau_{m},\xi_{m})\mathbbm{1}_{\{\tau_{m}\leqslant s\}}\prod_{\ell\geqslant1}\mathbbm{1}_{\{\tau_{m}\neq\rho_{\ell}(u)\}}
+V^{-}(s,X_{s}^{t,x;u,\beta(u)})\Big].
\end{align*}
In a similar way we can prove the reverse inequality, hence deducing the thesis.
\end{proof}

Now we prove the following corollary of the dynamic programming principle, which is concerned with the DPP for $s=t$. In particular, we prove that we can neglect multiple impulses, thanks to conditions (\ref{E:c}) and (\ref{E:chi}). Corollary \ref{C:DPPs=t} will be useful to prove that the two value functions are viscosity solutions to the HJBI equation.
\begin{corollary}
\label{C:DPPs=t}
Under assumptions \textup{\textbf{(H$_{b,\sigma}$)}}, \textup{\textbf{(H$_{f,g}$)}} and \textup{\textbf{(H$_{c,\chi}$)}}, given $(t,x)\in[0,T)\times\mathbb{R}^{n}$, we have:
\begin{align*}
V^{-}(t,x)=\inf_{\rho\in\mathcal{T}_{t,+\infty},\eta\in\mathcal{F}_{\rho}}
\sup_{\tau\in\mathcal{T}_{t,+\infty},\xi\in\mathcal{F}_{\tau}}
\mathbb{E}\Big[&-c(t,\xi)\mathbbm{1}_{\{\tau=t\}}\mathbbm{1}_{\{\rho=+\infty\}}+ \\
&\hspace{-2.5cm}+\chi(t,\eta)\mathbbm{1}_{\{\rho=t\}}+V^{-}(t,X_{t}^{t,x;\xi\mathbbm{1}_{[\tau,T]},\eta\mathbbm{1}_{[\rho,T]}})\Big],
\end{align*}
where $\mathcal{T}_{t,+\infty}$ is the set of $\mathbb{F}$-stopping times with values in $\{t,+\infty\}$. An analogous statement holds for the upper value function $V^{+}$.
\end{corollary}
\begin{proof}
We make the proof only for $V^{-}$, the other case being analogous. Consider the dynamic programming principle for $V^{-}$ with $s=t$:
\begin{align*}
V^{-}(t,x):=\inf_{\beta\in\mathcal{B}_{t,T}}\sup_{u\in\mathcal{U}_{t,T}}\mathbb{E}\Big[&
-\sum_{m\geqslant1}c(t,\xi_{m})\mathbbm{1}_{\{\tau_{m}=t\}}\prod_{\ell\geqslant1}\mathbbm{1}_{\{\tau_{m}\neq\rho_{\ell}(u)\}}+ \notag \\
&\hspace{1cm}+\sum_{\ell\geqslant1}\chi(t,\eta_{\ell}(u))\mathbbm{1}_{\{\rho_{\ell}(u)=t\}}+V^{-}(t,X_{t}^{t,x;u,\beta(u)})\Big].
\end{align*}
Let $\rho\in\mathcal{T}_{t,+\infty}$ and $\eta\in\mathcal{F}_{\rho}$, then consider the strategy $\beta(u)=\eta\mathbbm{1}_{[\rho,T]}$, for every $u\in\mathcal{U}_{t,T}$. Now fix $u=\sum_{m\geqslant1}\xi_{m}\mathbbm{1}_{[\tau_{m},T]}\in\mathcal{U}_{t,T}$ and define the impulse control $\hat{u}=\xi\mathbbm{1}_{[\tau,T]}$, where $\xi$ and $\tau$ are given by:
\[
\xi=\sum_{m\geqslant1}\xi_{m}\mathbbm{1}_{\{\tau_{m}=t\}},
\qquad\qquad
\tau=t\Big(1-\prod_{m\geqslant1}\mathbbm{1}_{\{\tau_{m}>t\}}\Big)+(+\infty)\prod_{m\geqslant1}\mathbbm{1}_{\{\tau_{m}>t\}}.
\]
Then $\tau\in\mathcal{T}_{t,+\infty}$ and $\xi\in\mathcal{F}_{\tau}$. Moreover $X_{t}^{t,x;u,\beta(u)}=X_{t}^{t,x;\xi\mathbbm{1}_{[\tau,T]},\eta\mathbbm{1}_{[\rho,T]}}$, $\mathbb{P}$-a.s.. Therefore, using condition (\ref{E:c}), we find
\begin{align*}
&\mathbb{E}\Big[-\sum_{m\geqslant1}c(t,\xi_{m})\mathbbm{1}_{\{\tau_{m}=t\}}\mathbbm{1}_{\{\rho=+\infty\}}+\chi(t,\eta)\mathbbm{1}_{\{\rho=t\}}
+V^{-}(t,X_{t}^{t,x;u,\eta\mathbbm{1}_{[\rho,T]}})\Big]\leqslant \\
&\quad\leqslant\mathbb{E}\Big[-c(t,\xi)\mathbbm{1}_{\{\tau=t\}}\mathbbm{1}_{\{\rho=+\infty\}}+\chi(t,\eta)\mathbbm{1}_{\{\rho=t\}}
+V^{-}(t,X_{t}^{t,x;\xi\mathbbm{1}_{[\tau,T]},\eta\mathbbm{1}_{[\rho,T]}})\Big].
\end{align*}
As a consequence we have
\begin{align*}
V^{-}(t,x)\leqslant\inf_{\rho\in\mathcal{T}_{t,+\infty},\eta\in\mathcal{F}_{\rho}}
\sup_{\tau\in\mathcal{T}_{t,+\infty},\xi\in\mathcal{F}_{\tau}}
\mathbb{E}\Big[&-c(t,\xi)\mathbbm{1}_{\{\tau=t\}}\mathbbm{1}_{\{\rho=+\infty\}}+ \\
&\hspace{-2.5cm}+\chi(t,\eta)\mathbbm{1}_{\{\rho=t\}}+V^{-}(t,X_{t}^{t,x;\xi\mathbbm{1}_{[\tau,T]},\eta\mathbbm{1}_{[\rho,T]}})\Big].
\end{align*}
In a similar way we can prove the reverse inequality, therefore deducing the thesis.
\end{proof}

To prove that the value functions are viscosity solutions to the HJBI equation, we need the dynamic programming principle also for stopping times which assume a countable number of values. This result is a simple extension of the dynamic programming principle for deterministic times and it is presented in the following lemma.
\begin{lemma}
\label{L:DPPtau}
Under assumptions \textup{\textbf{(H$_{b,\sigma}$)}}, \textup{\textbf{(H$_{f,g}$)}} and \textup{\textbf{(H$_{c,\chi}$)}}, given $0\leqslant t\leqslant s<T$, $x\in\mathbb{R}^{n}$ and a $[t,s]$-valued $\mathbb{F}$-stopping time $\tau$, which assumes a countable number of values, we have:
\begin{align*}
V^{-}(t,x):=\inf_{\beta\in\mathcal{B}_{t,T}}\sup_{u\in\mathcal{U}_{t,T}}\mathbb{E}\Big[\int_{t}^{\tau}f(X_{r}^{t,x;u,\beta(u)})\textup{d}r&
+\sum_{\ell\geqslant1}\chi(\rho_{\ell}(u),\eta_{\ell}(u))\mathbbm{1}_{\{\rho_{\ell}(u)\leqslant\tau\}} \\
&\hspace{-4cm}-\sum_{m\geqslant1}c(\tau_{m},\xi_{m})\mathbbm{1}_{\{\tau_{m}\leqslant \tau\}}\prod_{\ell\geqslant1}\mathbbm{1}_{\{\tau_{m}\neq\rho_{\ell}(u)\}}+V^{-}(\tau,X_{\tau}^{t,x;u,\beta(u)})\Big].
\end{align*}
An analogous statement holds for the upper value function $V^{+}$.
\end{lemma}
\begin{proof} Let $\tau=\sum_{j\geqslant1}t_{j}\mathbbm{1}_{B_{j}}$, with $t_{j}\in[t,s]$ and $B_{j}\in\mathcal{F}_{t_{j}}$. The proof is completely similar to the proof of the dynamic programming principle for deterministic times (Theorem \ref{T:DPP}). Therefore we focus only on the main point, that corresponds to inequality (\ref{E:DPPproof2}), which now becomes:
\begin{align*}
\mathbb{E}\big[V^{-}(\tau,X_{\tau}^{t,x;u,\beta^{1,\varepsilon}(u)})\big]
&=\mathbb{E}\Big[\sum_{j\geqslant1}V^{-}(\tau,X_{\tau}^{t,x;u,\beta^{1,\varepsilon}(u)})\mathbbm{1}_{B_{j}}\Big]\geqslant \\
&\geqslant\mathbb{E}\Big[\sum_{j\geqslant1}J\big(t_{j},X_{t_{j}}^{t,x;u,\beta^{1,\varepsilon}(u)};u_{[t_{j},T]},\beta^{2,\varepsilon,j}(u)\big)\mathbbm{1}_{B_{j}}\Big]-\varepsilon= \\
&=\mathbb{E}\Big[J\Big(\tau,X_{\tau}^{t,x;u,\beta^{1,\varepsilon}(u)};u_{[\tau,T]},\sum_{j\geqslant1}\beta^{2,\varepsilon,j}(u)\mathbbm{1}_{B_{j}}\Big)\Big]-\varepsilon,
\end{align*}
where $\beta^{2,\varepsilon,j}\in\bar{\mathcal{B}}_{t_{j},T}$, for every $j\geqslant1$. We define the strategy $\beta^{2,\varepsilon}:=\sum_{j\geqslant1}\beta^{2,\varepsilon,j}\mathbbm{1}_{B_{j}}$ and we conclude as in the proof of Theorem \ref{T:DPP}.
\end{proof}

We end this section with a technical lemma, which will be useful to prove that the two value functions satisfy, in the viscosity sense, the terminal condition to the HJBI equation.
\begin{lemma}
\label{L:DPPT}
Suppose that assumptions \textup{\textbf{(H$_{b,\sigma}$)}}, \textup{\textbf{(H$_{f,g}$)}} and \textup{\textbf{(H$_{c,\chi}$)}} hold true, then for every $(t,x)\in[0,T)\times\mathbb{R}^{n}$ we have
\begin{align}
\label{E:VtildeDPP}
V^{-}(t,x)=&\inf_{\rho\in\mathcal{T}_{t,+\infty},\eta\in\mathcal{F}_{\rho}}
\sup_{\tau\in\mathcal{T}_{t,+\infty},\xi\in\mathcal{F}_{\tau}}
\mathbb{E}\Big[\Big(-c(t,\xi)\mathbbm{1}_{\{\tau=t,\rho=+\infty\}}+ \\
&+\chi(t,\eta)\mathbbm{1}_{\{\rho=t\}}
+V^{-}(t,X_{t}^{t,x;\xi\mathbbm{1}_{[\tau,T]},\eta\mathbbm{1}_{[\rho,T]}})\Big)\big(1-\mathbbm{1}_{\{\tau=+\infty,\rho=+\infty\}}\big)+ \notag \\
&+\Big(\int_{t}^{T}f(X_{r}^{t,x;u_{0},v_{0}})\textup{d}r+g(X_{T}^{t,x;u_{0},v_{0}})\Big)\mathbbm{1}_{\{\tau=+\infty,\rho=+\infty\}}\Big], \notag
\end{align}
where $u_{0}$ and $v_{0}$ are the controls with no impulses. An analogous statement holds true for the upper value function $V^{+}$.
\end{lemma}
\begin{proof}
We make the proof only for $V^{-}$, the other case being analogous. Let $\tau\in\mathcal{T}_{t,+\infty}$, $\xi\in\mathcal{F}_{\tau}$ and $\varepsilon>0$, then there exist $\rho^{1,\varepsilon}\in\mathcal{T}_{t,+\infty}$ and $\eta^{1,\varepsilon}\in\mathcal{F}_{\rho^{1,\varepsilon}}$, such that the right hand side of equation (\ref{E:VtildeDPP}) is greater than or equal to the following expression:
\begin{align}
\label{E:DPPTproof}
&\mathbb{E}\Big[\Big(-c(t,\xi)\mathbbm{1}_{\{\tau=t,\rho^{1,\varepsilon}=+\infty\}}
+V^{-}(t,X_{t}^{t,x;\xi\mathbbm{1}_{[\tau,T]},\eta^{1,\varepsilon}\mathbbm{1}_{[\rho^{1,\varepsilon},T]}})+ \\
&\qquad\qquad+\chi(t,\eta^{1,\varepsilon})\mathbbm{1}_{\{\rho^{1,\varepsilon}=t\}}\Big)\big(1-\mathbbm{1}_{\{\tau=+\infty,\rho^{1,\varepsilon}=+\infty\}}\big)
+\Big(\int_{t}^{T}f(X_{r}^{t,x;u_{0},v_{0}})\textup{d}r+ \notag \\
&\qquad\qquad+g(X_{T}^{t,x;u_{0},v_{0}})\Big)\mathbbm{1}_{\{\tau=+\infty,\rho^{1,\varepsilon}=+\infty\}}\Big]
-\varepsilon. \notag
\end{align}
Now fix $\hat{u}\in\bar{\mathcal{U}}_{t,T}$, then, proceeding as in the proof of the dynamic programming principle (Theorem \ref{T:DPP}), thanks to the regularity of $V^{-}$ and $J$ with respect to the state variable, we deduce the existence of a strategy $\beta^{2,\varepsilon}\in\bar{\mathcal{B}}_{t,T}$ such that
\[
\mathbb{E}\big[V^{-}(t,X_{t}^{t,x;\xi\mathbbm{1}_{[\tau,T]},\eta^{1,\varepsilon}\mathbbm{1}_{[\rho^{1,\varepsilon},T]}})\big]\geqslant
\mathbb{E}\big[J\big(t,X_{t}^{t,x;\xi\mathbbm{1}_{[\tau,T]},\eta^{1,\varepsilon}\mathbbm{1}_{[\rho^{1,\varepsilon},T]}};
\hat{u},\beta^{2,\varepsilon}(\hat{u})\big)\big]-\varepsilon.
\]
Finally, we define the control $u\in\mathcal{U}_{t,T}$ and the strategy $\beta^{\varepsilon}\in\mathcal{B}_{t,T}$ as follows:
\[
u=\big[\big(\xi\mathbbm{1}_{\{\tau=t\}}+u_{0}\mathbbm{1}_{\{\tau=+\infty\}}\big)\mathbbm{1}_{t}+\hat{u}\big]
(1-\mathbbm{1}_{\{\tau=+\infty,\rho^{1,\varepsilon}=+\infty\}})+u_{0}\mathbbm{1}_{\{\tau=+\infty,\rho^{1,\varepsilon}=+\infty\}}
\]
and
\begin{align*}
\beta^{\varepsilon}(\tilde{u})&=\big[\big(\eta^{1,\varepsilon}\mathbbm{1}_{\{\rho^{1,\varepsilon}=t\}}+v_{0}\mathbbm{1}_{\{\rho^{1,\varepsilon}=+\infty\}}\big)\mathbbm{1}_{t}
+\beta^{2,\varepsilon}(\tilde{u})\big](1-\mathbbm{1}_{\{\tau=+\infty,\rho^{1,\varepsilon}=+\infty\}})+ \\
&\quad+v_{0}\mathbbm{1}_{\{\tau=+\infty,\rho^{1,\varepsilon}=+\infty\}},
\end{align*}
for every $\tilde{u}\in\mathcal{U}_{t,T}$. Therefore, from (\ref{E:DPPTproof}) we find
\begin{align*}
&\inf_{\rho\in\mathcal{T}_{t,+\infty},\eta\in\mathcal{F}_{\rho}}
\sup_{\tau\in\mathcal{T}_{t,+\infty},\xi\in\mathcal{F}_{\tau}}
\mathbb{E}\Big[\Big(-c(t,\xi)\mathbbm{1}_{\{\tau=t,\rho=+\infty\}}+ \\
&+\chi(t,\eta)\mathbbm{1}_{\{\rho=t\}}
+V^{-}(t,X_{t}^{t,x;\xi\mathbbm{1}_{[\tau,T]},\eta\mathbbm{1}_{[\rho,T]}})\Big)\big(1-\mathbbm{1}_{\{\tau=+\infty,\rho=+\infty\}}\big)+ \\
&+\Big(\int_{t}^{T}f(X_{r}^{t,x;u_{0},v_{0}})\textup{d}r+g(X_{T}^{t,x;u_{0},v_{0}})\Big)\mathbbm{1}_{\{\tau=+\infty,\rho=+\infty\}}\Big]\geqslant J(t,x;u,\beta^{\varepsilon}(u))-2\varepsilon,
\end{align*}
from which we deduce that the following inequality holds:
\begin{align*}
&\inf_{\rho\in\mathcal{T}_{t,+\infty},\eta\in\mathcal{F}_{\rho}}
\sup_{\tau\in\mathcal{T}_{t,+\infty},\xi\in\mathcal{F}_{\tau}}
\mathbb{E}\Big[\Big(-c(t,\xi)\mathbbm{1}_{\{\tau=t,\rho=+\infty\}}+ \\
&+\chi(t,\eta)\mathbbm{1}_{\{\rho=t\}}
+V^{-}(t,X_{t}^{t,x;\xi\mathbbm{1}_{[\tau,T]},\eta\mathbbm{1}_{[\rho,T]}})\Big)\big(1-\mathbbm{1}_{\{\tau=+\infty,\rho=+\infty\}}\big)+ \\
&+\Big(\int_{t}^{T}f(X_{r}^{t,x;u_{0},v_{0}})\textup{d}r+g(X_{T}^{t,x;u_{0},v_{0}})\Big)\mathbbm{1}_{\{\tau=+\infty,\rho=+\infty\}}\Big]\geqslant V^{-}(t,x).
\end{align*}
In a similar way we can prove the reverse inequality and thus we get the thesis.
\end{proof}

\section{Hamilton-Jacobi-Bellman-Isaacs Equation}
\label{S:HJBI}
In the present section we give the definition of viscosity solution to the HJBI equation (\ref{E:HJBI}) and we prove that the two value functions are viscosity solutions to this equation.
\begin{definition}
A lower (resp., upper) semicontinuous function $V:[0,T]\times\mathbb{R}^{n}\rightarrow\mathbb{R}$ is called a viscosity supersolution (resp., subsolution) to the HJBI equation~(\ref{E:HJBI}) if
\begin{enumerate}[\upshape (i)]
    \item for every $(t,x)\in[0,T)\times\mathbb{R}^{n}$ and $\varphi\in C^{1,2}([0,T)\times\mathbb{R}^{n})$, such that $(t,x)$ is a local minimum (resp.,
        maximum) of $V-\varphi$, we have
        \begin{align}
        \label{E:HJBIvisc}
        &\max\Big\{\min\Big[-\dfrac{\partial\varphi}{\partial t}(t,x)-\mathcal{L}\varphi(t,x)-f(x),V(t,x)-\mathcal{H}_{\sup}^{c}V(t,x)\Big], \\
        &\hspace{5cm}V(t,x)-\mathcal{H}_{\inf}^{\chi}V(t,x)\Big\}\geqslant0 \quad \text{\textup{(}resp.,
        $\leqslant0$\textup{)}}. \notag
        \end{align}
    \item for every $x\in\mathbb{R}^{n}$ we have
        \begin{align}
        \label{E:HJBIfinal}
        &\max\Big\{\min\Big[V(T,x)-g(x),V(T,x)-\mathcal{H}_{\sup}^{c}V(T,x)\Big], \\
        &\hspace{5cm}V(T,x)-\mathcal{H}_{\inf}^{\chi}V(T,x)\Big\}\geqslant0 \quad \text{\textup{(}resp., $\leqslant0$\textup{)}}.
        \notag
        \end{align}
\end{enumerate}
A locally bounded function $V\colon[0,T)\times\mathbb{R}^{n}\rightarrow\mathbb{R}$ is a viscosity solution to the HJBI equation (\ref{E:HJBI}) if its lower semicontinuous envelope $V_{*}$ is a viscosity supersolution and its upper semicontinuous envelope $V^{*}$ is a viscosity subsolution. $V_{*}$ and $V^{*}$ are given by:
\[
V_{*}(t,x):=\liminf_{(s,y)\rightarrow(t,x),\,s<T}V(s,y), \qquad\qquad V^{*}(t,x):=\limsup_{(s,y)\rightarrow(t,x),\,s<T}V(s,y),
\]
for every $(t,x)\in[0,T]\times\mathbb{R}^{n}$.
\end{definition}

\begin{rmk}
\label{V*}
\textup{The lower and upper semicontinuous envelopes of the restriction of $V^{-}$ to $[0,T)\times\mathbb{R}^{n}$ coincide on $[0,T]\times\mathbb{R}^{n}$ and are equal to the ($1/2$-H\"older continuous in time and Lipschitz continuous in the state variable) extension of $V^{-}$, see Remark~\ref{R:ContExt}. An analogous remark applies to the upper value function $V^{+}$.}
\end{rmk}

To prove that the two value functions are viscosity solution to the HJBI equation~(\ref{E:HJBI}), we begin with the following lemma.
\begin{lemma}
\label{L:Bounds}
Under assumptions \textup{\textbf{(H$_{b,\sigma}$)}}, \textup{\textbf{(H$_{f,g}$)}} and \textup{\textbf{(H$_{c,\chi}$)}}, the lower and upper value functions satisfy the following equation:
\begin{equation}
\max\Big\{\min\Big[0,V(t,x)-\mathcal{H}_{\sup}^{c}V(t,x)\Big],V(t,x)-\mathcal{H}_{\inf}^{\chi}V(t,x)\Big\}=0, \notag
\end{equation}
for all $t\in[0,T)$ and $x\in\mathbb{R}^{n}$.
\end{lemma}
\begin{proof}
We prove the lemma only for $V^{-}$, since the proof for $V^{+}$ is analogous. When $s=t$ in the dynamic programming principle for $V^{-}$, we have, thanks to Corollary \ref{C:DPPs=t}:
\begin{align*}
V^{-}(t,x)=\inf_{\rho\in\mathcal{T}_{t,+\infty},\eta\in\mathcal{F}_{\rho}}
\sup_{\tau\in\mathcal{T}_{t,+\infty},\xi\in\mathcal{F}_{\tau}}
\mathbb{E}\Big[&-c(t,\xi)\mathbbm{1}_{\{\tau=t\}}\mathbbm{1}_{\{\rho=+\infty\}}+ \\
&\hspace{-2.5cm}+\chi(t,\eta)\mathbbm{1}_{\{\rho=t\}}+V^{-}(t,X_{t}^{t,x;\xi\mathbbm{1}_{[\tau,T]},\eta\mathbbm{1}_{[\rho,T]}})\Big].
\end{align*}
Note that $X_{t}^{t,x;\xi\mathbbm{1}_{[\tau,T]},\eta\mathbbm{1}_{[\rho,T]}}=x+\xi\mathbbm{1}_{\{\tau=t\}}\mathbbm{1}_{\{\rho=+\infty\}}+\eta\mathbbm{1}_{\{\rho=t\}}$, $\mathbb{P}$-a.s.. It is simple to show that to attain the optimum we need to consider only deterministic quadruples: $(\rho,z,\tau,y)\in\{t,+\infty\}\times\mathscr{V}\times\{t,+\infty\}\times\mathscr{U}$. Consequently we have:
\begin{align*}
V^{-}(t,x)=\inf_{\rho\in\{t,+\infty\},z\in\mathscr{V}}
\sup_{\tau\in\{t,+\infty\},y\in\mathscr{U}}
\Big\{&-c(t,y)\mathbbm{1}_{\{\tau=t\}}\mathbbm{1}_{\{\rho=+\infty\}}+ \\
&\hspace{-2.5cm}+\chi(t,z)\mathbbm{1}_{\{\rho=t\}}+V^{-}(t,X_{t}^{t,x;y\mathbbm{1}_{[\tau,T]},z\mathbbm{1}_{[\rho,T]}})\Big\}.
\end{align*}
Therefore, rearranging the terms, the above equation can be written as follows:
\begin{align*}
\sup_{\rho\in\{t,+\infty\},z\in\mathscr{V}}
\inf_{\tau\in\{t,+\infty\},y\in\mathscr{U}}
&\Big\{\Big[V^{-}(t,x)-\big(V^{-}(t,x+y)-c(t,y)\big)\Big]\mathbbm{1}_{\{\tau=t,\rho=+\infty\}}+ \\
&+\Big[V^{-}(t,x)-\big(V^{-}(t,x+z)+\chi(t,z)\big)\Big]\mathbbm{1}_{\{\rho=t\}}\Big\}=0.
\end{align*}
We can write this as
\begin{align*}
\sup_{\rho\in\{t,+\infty\},z\in\mathscr{V}}
&\Big\{\min_{\tau\in\{t,+\infty\}}\Big\{\big(V^{-}(t,x)-\mathcal{H}_{\sup}^{c}V^{-}(t,x)\big)\mathbbm{1}_{\{\tau=t,\rho=+\infty\}}\Big\}+ \\
&+\Big[V^{-}(t,x)-\big(V^{-}(t,x+z)+\chi(t,z)\big)\Big]\mathbbm{1}_{\{\rho=t\}}\Big\}=0,
\end{align*}
which becomes
\begin{align*}
\max_{\rho\in\{t,+\infty\}}
&\Big\{\min\Big[0,V^{-}(t,x)-\mathcal{H}_{\sup}^{c}V^{-}(t,x)\Big]\mathbbm{1}_{\{\rho=+\infty\}}+ \\
&+\Big[V^{-}(t,x)-\mathcal{H}_{\inf}^{\chi}V^{-}(t,x)\Big]\mathbbm{1}_{\{\rho=t\}}\Big\}=0.
\end{align*}
Now we can easily derive the thesis.
\end{proof}
\begin{rmk}
\label{R:Bounds}
\textup{From Lemma \ref{L:Bounds} we deduce that $V^{-}\leqslant\mathcal{H}_{\inf}^{\chi}V^{-}$ on $[0,T)\times\mathbb{R}^{n}$. Moreover, when $V^{-}<\mathcal{H}_{\inf}^{\chi}V^{-}$, then $\mathcal{H}_{\sup}^{c}V^{-}\leqslant V^{-}$ and $\mathcal{H}_{\sup}^{c}V^{-}<\mathcal{H}_{\inf}^{\chi}V^{-}$. Therefore, we can interpret $\mathcal{H}_{\sup}^{c}V^{-}$ as a lower obstacle and $\mathcal{H}_{\inf}^{\chi}V^{-}$ as an upper obstacle. They are implicit obstacles, in the sense that they depend on $V^{-}$. Clearly the same remark applies to $V^{+}$.}
\end{rmk}

Now we prove that the two value functions satisfy, in the viscosity sense, the terminal condition.
\begin{lemma}
\label{L:TC}
Under assumptions \textup{\textbf{(H$_{b,\sigma}$)}}, \textup{\textbf{(H$_{f,g}$)}} and \textup{\textbf{(H$_{c,\chi}$)}}, the lower and upper value functions are viscosity solutions to (\ref{E:HJBIfinal}).
\end{lemma}
\begin{proof}
We make the proof only for $V^{-}$, the other case being analogous. Firstly, we prove the supersolution property. We have to prove that, for every $x\in\mathbb{R}^{n}$, the following inequality holds:
\begin{align}
\label{E:BoundsT}
\max\Big\{\min\Big[V_{*}^{-}(T,x)-g(x),&\,V_{*}^{-}(T,x)-\mathcal{H}_{\sup}^{c}V_{*}^{-}(T,x)\Big], \\
&V_{*}^{-}(T,x)-\mathcal{H}_{\inf}^{\chi}V_{*}^{-}(T,x)\Big\}\geqslant0. \notag
\end{align}
Thanks to Lemma \ref{L:DPPT} and the fact that $V^{-}$ and $V_{*}^{-}$ coincide on $[0,T)\times\mathbb{R}^{n}$, we have for every $t\in[0,T)$:
\begin{align*}
V_{*}^{-}(t,x)&=\inf_{\rho\in\mathcal{T}_{t,+\infty},\eta\in\mathcal{F}_{\rho}}
\sup_{\tau\in\mathcal{T}_{t,+\infty},\xi\in\mathcal{F}_{\tau}}
\mathbb{E}\Big[\Big(-c(t,\xi)\mathbbm{1}_{\{\tau=t,\rho=+\infty\}}+\chi(t,\eta)\mathbbm{1}_{\{\rho=t\}}+ \\
&\qquad+V_{*}^{-}(t,X_{t}^{t,x;\xi\mathbbm{1}_{[\tau,T]},\eta\mathbbm{1}_{[\rho,T]}})\Big)\big(1-\mathbbm{1}_{\{\tau=+\infty,\rho=+\infty\}}\big)+ \\
&\qquad+\Big(\int_{t}^{T}f(X_{r}^{t,x;u_{0},v_{0}})\textup{d}r+g(X_{T}^{t,x;u_{0},v_{0}})\Big)\mathbbm{1}_{\{\tau=+\infty,\rho=+\infty\}}\Big].
\end{align*}
Thanks to the boundedness of $f$ we find (in the sequel the letter $C$ stands for a positive constant, independent of $t$, whose value may change from line to line):
\[
\mathbb{E}\Big[\Big|\int_{t}^{T}f(X_{r}^{t,x;u_{0},v_{0}})\textup{d}r\Big|\Big]\leqslant C(T-t).
\]
Furthermore, using the Lipschitzianity of $g$ and the fact that $u_{0}$ and $v_{0}$ are the control with no impulses, we deduce the following standard result:
\[
\mathbb{E}\big[\big|g(X_{T}^{t,x;u_{0},v_{0}})-g(x)\big|\big]\leqslant C(1+|x|)|T-t|^{\frac{1}{2}}.
\]
As a consequence we get:
\begin{align*}
V_{*}^{-}(t,x)&\geqslant\inf_{\rho\in\mathcal{T}_{t,+\infty},\eta\in\mathcal{F}_{\rho}}
\sup_{\tau\in\mathcal{T}_{t,+\infty},\xi\in\mathcal{F}_{\tau}}
\mathbb{E}\Big[\Big(-c(t,\xi)\mathbbm{1}_{\{\tau=t,\rho=+\infty\}}+\chi(t,\eta)\mathbbm{1}_{\{\rho=t\}}+ \\
&\qquad+V_{*}^{-}(t,X_{t}^{t,x;\xi\mathbbm{1}_{[\tau,T]},\eta\mathbbm{1}_{[\rho,T]}})\Big)\big(1-\mathbbm{1}_{\{\tau=+\infty,\rho=+\infty\}}\big)+ \\
&\qquad+g(x)\mathbbm{1}_{\{\tau=+\infty,\rho=+\infty\}}\Big]-C(1+|x|)|T-t|^{\frac{1}{2}}.
\end{align*}
Now, proceeding as in the proof of Lemma \ref{L:Bounds}, we end up with the following inequality:
\begin{align*}
\max\Big\{\min\Big[V_{*}^{-}(t,x)-g(x),&\,V_{*}^{-}(t,x)-\mathcal{H}_{\sup}^{c}V_{*}^{-}(t,x)\Big], \\
&V_{*}^{-}(t,x)-\mathcal{H}_{\inf}^{\chi}V_{*}^{-}(t,x)\Big\}\geqslant-C(1+|x|)|T-t|^{\frac{1}{2}}.
\end{align*}
Using the $1/2$-H\"older continuity in time of $V_{*}^{-}$, $c$ and $\chi$, uniformly in the other variable, we deduce that also the left-hand side is $1/2$-H\"older continuous in time, uniformly with respect to $x$. Therefore we have:
\begin{align*}
&\max\Big\{\min\Big[V_{*}^{-}(T,x)-g(x),V_{*}^{-}(T,x)-\mathcal{H}_{\sup}^{c}V_{*}^{-}(T,x)\Big],V_{*}^{-}(T,x)-\mathcal{H}_{\inf}^{\chi}V_{*}^{-}(T,x)\Big\}+ \\
&\quad+C|T-t|^{\frac{1}{2}}\geqslant
\max\Big\{\min\Big[V_{*}^{-}(t,x)-g(x),V_{*}^{-}(t,x)-\mathcal{H}_{\sup}^{c}V_{*}^{-}(t,x)\Big], \\
&\quad V_{*}^{-}(t,x)-\mathcal{H}_{\inf}^{\chi}V_{*}^{-}(t,x)\Big\}.
\end{align*}
Hence we see that (\ref{E:BoundsT}) holds. In an analogous manner we can prove the subsolution property.
\end{proof}

Now we are ready to state one of our main results.
\begin{theorem}
\label{T:Exist}
Under assumptions \textup{\textbf{(H$_{b,\sigma}$)}}, \textup{\textbf{(H$_{f,g}$)}} and \textup{\textbf{(H$_{c,\chi}$)}}, the lower and upper value functions are viscosity solutions to the HJBI equation (\ref{E:HJBI}).
\end{theorem}
\begin{proof}
We give the proof for the lower value function $V^{-}$, the other case being analogous. Thanks to Lemma \ref{L:TC} we have that $V^{-}$ satisfies in the viscosity sense the terminal condition. As a consequence, we only have to address (\ref{E:HJBIvisc}).

We know from Proposition \ref{P:Cont-x} and Proposition \ref{P:Cont-t} that $V^{-}$ is continuous on $[0,T)\times\mathbb{R}^{n}$, therefore $V^{-}$ is equal to its lower semicontinuous envelope and to its upper semicontinuous envelope on $[0,T)\times\mathbb{R}^{n}$. We begin by proving that $V^{-}$ is a viscosity supersolution. Thanks to Lemma \ref{L:Bounds}, it is enough to show that given $(t_{0},x_{0})\in[0,T)\times\mathbb{R}^{n}$ such that
\[
V^{-}(t_{0},x_{0})-\mathcal{H}_{\sup}^{c}V^{-}(t_{0},x_{0})\geqslant0 \qquad \text{and} \qquad
V^{-}(t_{0},x_{0})-\mathcal{H}_{\inf}^{\chi}V^{-}(t_{0},x_{0})<0,
\]
then for every $\varphi\in C^{1,2}([0,T)\times\mathbb{R}^{n})$, such that $(t_{0},x_{0})$ is a local minimum of $V^{-}-\varphi$, we have
\[
-\frac{\partial\varphi}{\partial t}(t_{0},x_{0})-\mathcal{L}\varphi(t_{0},x_{0})-f(x_{0})\geqslant0.
\]
We may assume, without loss of generality, that
\[
V^{-}(t_{0},x_{0})=\varphi(t_{0},x_{0}).
\]
Let $\lambda>0$ be such that
\[
\lambda+V^{-}(t_{0},x_{0})=\mathcal{H}_{\inf}^{\chi}V^{-}(t_{0},x_{0})=\inf_{z\in\mathscr{V}}\big(V^{-}(t_{0},x_{0}+z)+\chi(t_{0},z)\big).
\]
From the regularity of $V^{-}$ we have (in the sequel the letter $C$ stands for a positive constant, whose value may change from line to line)
\[
\mathbb{E}[|V^{-}(s,X_{s}^{t_{0},x_{0}})-V^{-}(t_{0},x_{0})|]\leqslant C|s-t_{0}|^{\frac{1}{2}},
\]
with $X_{s}^{t_{0},x_{0}}=X_{s}^{t_{0},x_{0};u_{0},v_{0}}$, for all $s\in[t_{0},T]$, $\mathbb{P}$-a.s., where $u_{0}$ and $v_{0}$ are the controls with no impulses. Analogously, thanks to the $1/2$-H\"older continuity of $\chi$ with respect to time, uniformly with respect to the other variable, we have that $|\chi(s,z)-\chi(t_{0},z)|\leqslant C|s-t_{0}|^{1/2}$. Therefore, for every random variable $\eta$, $\mathcal{F}_{s}$-measurable and assuming values in $\mathscr{V}$, we find:
\begin{equation}
\label{E:EHinf}
\mathbb{E}[V^{-}(s,X_{s}^{t_{0},x_{0}})]\leqslant\mathbb{E}[V^{-}(s,X_{s}^{t_{0},x_{0}}+\eta)+\chi(s,\eta)]+C|s-t_{0}|^{\frac{1}{2}}-\lambda.
\end{equation}
Now, for every $\varepsilon>0$, using the dynamic programming principle for $V^{-}$, Lemma \ref{L:DPPtau}, with $s\in[t_{0},T)$ and $\tau$ a $[t_{0},s]$-valued $\mathbb{F}$-stopping time, which assumes a countable number of values, we have
\begin{align*}
&\varphi(t_{0},x_{0})=V^{-}(t_{0},x_{0})=
\inf_{\beta\in\mathcal{B}_{t_{0},T}}\sup_{u\in\mathcal{U}_{t_{0},T}}\mathbb{E}\Big[\int_{t_{0}}^{\tau}f(X_{r}^{t_{0},x_{0};u,\beta(u)})\textup{d}r
+V^{-}(\tau,X_{\tau}^{t_{0},x_{0};u,\beta(u)}) \\
&\qquad\quad-\sum_{m\geqslant1}c(\tau_{m},\xi_{m})\mathbbm{1}_{\{\tau_{m}\leqslant \tau\}}\prod_{\ell\geqslant1}\mathbbm{1}_{\{\tau_{m}\neq\rho_{\ell}(u)\}}+\sum_{\ell\geqslant1}\chi(\rho_{\ell}(u),\eta_{\ell}(u))\mathbbm{1}_{\{\rho_{\ell}(u)\leqslant \tau\}}\Big]\geqslant \\
&\qquad\quad\geqslant
\mathbb{E}\Big[\int_{t_{0}}^{\tau}f(X_{r}^{t_{0},x_{0};u_{0},v_{\varepsilon}})\textup{d}r +\sum_{\ell\geqslant1}\chi(\rho_{\ell}^{\varepsilon},\eta_{\ell}^{\varepsilon})\mathbbm{1}_{\{\rho_{\ell}^{\varepsilon}\leqslant \tau\}}+V^{-}(\tau,X_{\tau}^{t_{0},x_{0};u_{0},v_{\varepsilon}})\Big]-\varepsilon,
\end{align*}
for some $\beta_{\varepsilon}\in\mathcal{B}_{t_{0},T}$, with $\beta_{\varepsilon}(u_{0})=:v_{\varepsilon}=\sum_{\ell\geqslant}\eta_{\ell}^{\varepsilon}\mathbbm{1}_{[\rho_{\ell}^{\varepsilon},T]}\in\mathcal{V}_{t_{0},T}$. Our goal is to show that the following inequality holds true:
\begin{equation}
\label{E:ExistenceProof}
\varphi(t_{0},x_{0})\geqslant\mathbb{E}\Big[\int_{t_{0}}^{\tau}f(X_{r}^{t_{0},x_{0}})\textup{d}r+V^{-}(\tau,X_{\tau}^{t_{0},x_{0}})\Big]-\varepsilon,
\end{equation}
if we take $s$ sufficiently small. We exploit (\ref{E:EHinf}) to derive (\ref{E:ExistenceProof}).

Remember that the following identity holds:
\[
\sum_{\ell\geqslant1}\chi(\rho_{\ell}^{\varepsilon},\eta_{\ell}^{\varepsilon})\mathbbm{1}_{\{\rho_{\ell}^{\varepsilon}\leqslant \tau\}}
=\sum_{\ell=1}^{\mu_{t_{0},\tau}(v_{\varepsilon})}\chi(\rho_{\ell}^{\varepsilon},\eta_{\ell}^{\varepsilon}),
\]
where $\mu$ is given by equation (\ref{E:mu}). Now, using the boundedness of $b$ and $\sigma$, we get
\[
\mathbb{E}\Big[\Big|V^{-}(\tau,X_{\tau}^{t_{0},x_{0};u_{0},v_{\varepsilon}})
-V^{-}\Big(\tau,X_{\tau}^{t_{0},x_{0}}+\sum_{\ell=1}^{\mu_{t_{0},\tau}(v_{\varepsilon})}\eta_{\ell}^{\varepsilon}\Big)\Big|\Big]\leqslant C|s-t_{0}|^{\frac{1}{2}}\mathbb{E}[\mathbbm{1}_{\{\mu_{t_{0},\tau}(v_{\varepsilon})\geqslant1\}}].
\]
Furthermore, using (\ref{E:chi}), (\ref{E:c_chi_time}) and (\ref{E:EHinf}) we find
\begin{align*}
\mathbb{E}\Big[&\sum_{\ell=1}^{\mu_{t_{0},\tau}(v_{\varepsilon})}\chi(\rho_{\ell}^{\varepsilon},\eta_{\ell}^{\varepsilon})
+V^{-}\Big(\tau,X_{\tau}^{t_{0},x_{0}}+\sum_{\ell=1}^{\mu_{t_{0},\tau}(v_{\varepsilon})}\eta_{\ell}^{\varepsilon}\Big)\Big]\geqslant
\mathbb{E}\Big[\chi\Big(\tau,\sum_{\ell=1}^{\mu_{t_{0},\tau}(v_{\varepsilon})}\eta_{\ell}^{\varepsilon}\Big)
\mathbbm{1}_{\{\mu_{t_{0},\tau}(v_{\varepsilon})\geqslant1\}}+ \\
&+V^{-}\Big(\tau,X_{\tau}^{t_{0},x_{0}}+\sum_{\ell=1}^{\mu_{t_{0},\tau}(v_{\varepsilon})}\eta_{\ell}^{\varepsilon}\Big)
-V^{-}(\tau,X_{\tau}^{t_{0},x_{0}})+V^{-}(\tau,X_{\tau}^{t_{0},x_{0}})\Big]\geqslant
\mathbb{E}\big[V^{-}(\tau,X_{\tau}^{t_{0},x_{0}})+ \\
&+(\lambda-C|s-t_{0}|^{\frac{1}{2}})\mathbbm{1}_{\{\mu_{t_{0},\tau}(v_{\varepsilon})\geqslant1\}}\big].
\end{align*}
Hence
\[
\varphi(t_{0},x_{0})\geqslant\mathbb{E}\Big[\int_{t_{0}}^{\tau}f(X_{r}^{t_{0},x_{0};u_{0},v_{\varepsilon}})\textup{d}r
+V^{-}(\tau,X_{\tau}^{t_{0},x_{0}})+(\lambda-C|s-t_{0}|^{\frac{1}{2}})\mathbbm{1}_{\{\mu_{t_{0},\tau}(v_{\varepsilon})\geqslant1\}}\Big]-\varepsilon.
\]
Using the boundedness of $f$ we deduce
\begin{align*}
\varphi(t_{0},x_{0})\geqslant\mathbb{E}\Big[\int_{t_{0}}^{\tau}f(X_{r}^{t_{0},x_{0}})\textup{d}r
+V^{-}(\tau,X_{\tau}^{t_{0},x_{0}})+\big(\lambda&-C|s-t_{0}|^{\frac{1}{2}} \\
&-C|s-t_{0}|\big)\mathbbm{1}_{\{\mu_{t_{0},\tau}(v_{\varepsilon})\geqslant1\}}\Big]-\varepsilon.
\end{align*}
Therefore there exists $\bar{s}\in(t_{0},T)$ such that for $s\in[t_{0},\bar{s}\,]$ we have
\[
\lambda-C|s-t_{0}|^{\frac{1}{2}}-C|s-t_{0}|\geqslant0.
\]
Consequently, for every $[t_{0},\bar{s}\,]$-valued $\mathbb{F}$-stopping time $\tau$, which assumes a countable number of values, we deduce that inequality (\ref{E:ExistenceProof}) holds.

Since $(t_{0},x_{0})$ is a local minimum of $V^{-}-\varphi$, there exists $\delta>0$ such that
\[
V^{-}(t,x)\geqslant\varphi(t,x), \qquad\qquad |t-t_{0}|\leqslant\delta,\,|x-x_{0}|\leqslant\delta.
\]
Moreover there exists a positive constant $C$ such that, for every $t\in[t_{0},T)$, we have
\[
\mathbb{E}[|X_{t}^{t_{0},x_{0}}-x_{0}|]\leqslant C|t-t_{0}|^{\frac{1}{2}}.
\]
Consequently there exists a sequence $t_{n}\downarrow t_{0}$ such that $X_{t_{n}}^{t_{0},x_{0}}\rightarrow x_{0}$, $\mathbb{P}$-a.s., as $n$ tends to infinity. We may assume that $t_{n}\leqslant\bar{s}$, for every $n\geqslant1$. Now define the following sets $B_{n}\subset\Omega$:
\[
B_{n}=\big\{|X_{t_{m}}^{t_{0},x_{0}}-x_{0}|\leqslant\delta,\,\forall m\geqslant n\big\}.
\]
Then $B_{n}\uparrow B:=\cup_{n\geqslant1}B_{n}$, with $\mathbb{P}(B)=1$. Furthermore, for every $m\geqslant1$, let introduce the stopping time
\[
\tau_{m}=\sum_{n=1}^{\infty}t_{n+m}\mathbbm{1}_{\{B_{n}\backslash B_{n-1}\}},
\]
where $B_{0}$ is the empty set. Then $\tau_{m}\downarrow t_{0}$, $\mathbb{P}$-a.s., moreover $\tau_{m}\leqslant t_{m}$. Inserting $\tau_{m}$ in (\ref{E:ExistenceProof}), we find
\begin{align*}
\varphi(t_{0},x_{0})&\geqslant\mathbb{E}\Big[\int_{t_{0}}^{\tau_{m}}f(X_{r}^{t_{0},x_{0}})\textup{d}r+V^{-}(\tau_{m},X_{\tau_{m}}^{t_{0},x_{0}})\Big]-\varepsilon
\geqslant \\
&\geqslant\mathbb{E}\Big[\int_{t_{0}}^{\tau_{m}}f(X_{r}^{t_{0},x_{0}})\textup{d}r+\varphi(\tau_{m},X_{\tau_{m}}^{t_{0},x_{0}})\Big]-\varepsilon.
\end{align*}
An application of It\^o's formula yields
\[
0\geqslant\mathbb{E}\Big[\int_{t_{0}}^{\tau_{m}}\Big(\frac{\partial\varphi}{\partial t}(r,X_{r}^{t_{0},x_{0}})
+\mathcal{L}\varphi(r,X_{r}^{t_{0},x_{0}})+f(X_{r}^{t_{0},x_{0}})\Big)\textup{d}r\Big]-\varepsilon.
\]
Taking $\varepsilon=\bar{\varepsilon}(t_{m}-t_{0})$, with $\bar{\varepsilon}>0$, and dividing both sides by $t_{m}-t_{0}$, we get
\[
\bar{\varepsilon}\geqslant\mathbb{E}\Big[\frac{1}{t_{m}-t_{0}}\int_{t_{0}}^{t_{m}}\Big(\frac{\partial\varphi}{\partial t}(r,X_{r}^{t_{0},x_{0}})
+\mathcal{L}\varphi(r,X_{r}^{t_{0},x_{0}})+f(X_{r}^{t_{0},x_{0}})\Big)\mathbbm{1}_{\{r\leqslant\tau_{m}\}}\textup{d}r\Big].
\]
Sending $m$ to $\infty$, we end up with
\[
\bar{\varepsilon}\geqslant\frac{\partial\varphi}{\partial t}(t_{0},x_{0})+\mathcal{L}\varphi(t_{0},x_{0})+f(x_{0}).
\]
Therefore we find
\[
-\frac{\partial\varphi}{\partial t}(t_{0},x_{0})-\mathcal{L}\varphi(t_{0},x_{0})-f(x_{0})\geqslant0,
\]
which is the supersolution property. The subsolution property is proved analogously.
\end{proof}

We end this section providing another definition of viscosity solution to the HJBI equation by means of jets, needed later in the proof of the Comparison Theorem (Theorem \ref{T:Uniq}). In the following definition we denote by $\mathbb{S}(n)$ the set of symmetric matrices of dimension $n$.
\begin{definition}
Let $V:[0,T]\times\mathbb{R}^{n}\rightarrow\mathbb{R}$ be a lower semicontinuous function, then we denote by $J^{2,-}V(t,x)$ the parabolic subjet of $V$ at $(t,x)\in[0,T)\times\mathbb{R}^{n}$ as the set of triples $(p,q,M)\in\mathbb{R}\times\mathbb{R}^{n}\times\mathbb{S}(n)$ such that
\begin{align*}
V(s,y)\geqslant&\,V(t,x)+p(s-t)+\langle q,y-x\rangle+\tfrac{1}{2}\langle M(y-x),y-x\rangle+ \\
&+o(|s-t|+|y-x|^{2}),
\end{align*}
as $s\rightarrow t$ $(s\rightarrow t^{+}$, when $t=0)$ and $y\rightarrow x$. We also introduce the parabolic limiting subjet of $V$ at $(t,x)\in[0,T)\times\mathbb{R}^{n}\colon$
\begin{align*}
\bar{J}^{2,-}V(t,x)=&\,\{(p,q,M)\in\mathbb{R}\times\mathbb{R}^{n}\times\mathbb{S}(n):\exists(t_{n},x_{n},p_{n},q_{n},M_{n})\in[0,T)\times \\
&\quad\times\mathbb{R}^{n}\times\mathbb{R}\times\mathbb{R}^{n}\times\mathbb{S}(n)\text{ such that }(p_{n},q_{n},M_{n})\in J^{2,-}V(t_{n},x_{n}) \\
&\quad\text{and }(t_{n},x_{n},V(t_{n},x_{n}),p_{n},q_{n},M_{n})\rightarrow(t,x,V(t,x),p,q,M)\}.
\end{align*}
When $V$ is an upper semicontinuous function on $[0,T]\times\mathbb{R}^{n}$, we define the parabolic superjet $J^{2,+}V(t,x)$ and the parabolic limiting superjet $\bar{J}^{2,+}V(t,x)$ of $V$ at $(t,x)\in[0,T)\times\mathbb{R}^{n}$ by
\[
J^{2,+}V(t,x)=-J^{2,-}(-V)(t,x) \qquad \text{ and } \qquad \bar{J}^{2,+}V(t,x)=-\bar{J}^{2,-}(-V)(t,x).
\]
\end{definition}
Then we have the following result, whose standard proof is omitted.
\begin{lemma}
\label{L:Jets}
Let $V:[0,T]\times\mathbb{R}^{n}\rightarrow\mathbb{R}$ be a lower (resp., upper) semicontinuous function. Then $V$ is a viscosity supersolution (resp., subsolution) to the HJBI equation (\ref{E:HJBI}) if and only if
\begin{enumerate}[\upshape (i)]
    \item for every $(t,x)\in[0,T)\times\mathbb{R}^{n}$ and $(p,q,M)\in\bar{J}^{2,-}V(t,x)$ $(\text{resp., }\bar{J}^{2,+}V(t,x))$ we have
        \begin{align}
        &\max\Big\{\min\Big[p-\langle b(x),q\rangle-\tfrac{1}{2}\textup{tr}\big[(\sigma\sigma')(x)M\big]-f(x),V(t,x) \\
        &\hspace{3cm}-\mathcal{H}_{\sup}^{c}V(t,x)\Big],V(t,x)-\mathcal{H}_{\inf}^{\chi}V(t,x)\Big\}\geqslant0 \quad \text{\textup{(}resp.,
        $\leqslant0$\textup{)}}. \notag
        \end{align}
    \item for every $x\in\mathbb{R}^{n}$ we have
        \begin{align}
        &\max\Big\{\min\Big[V(T,x)-g(x),V(T,x)-\mathcal{H}_{\sup}^{c}V(T,x)\Big], \\
        &\hspace{5cm}V(T,x)-\mathcal{H}_{\inf}^{\chi}V(T,x)\Big\}\geqslant0 \quad \text{\textup{(}resp., $\leqslant0$\textup{)}}.
        \notag
        \end{align}
\end{enumerate}
\end{lemma}

\section{Uniqueness}
\label{S:Uniqueness}
We prove that the HJBI equation (\ref{E:HJBI}) admits a unique viscosity solution. As a consequence, the lower and upper value functions coincide, since they are both viscosity solutions to (\ref{E:HJBI}). Hence the stochastic differential game has a value.

Before proving the Comparison Theorem, we need the following two technical lemmas.
\begin{lemma}
\label{L:Uniq}
Let $V,U\colon[0,T]\times\mathbb{R}^{n}\rightarrow\mathbb{R}$ be a viscosity subsolution and a viscosity supersolution to the HJBI equation (\ref{E:HJBI}), respectively, and suppose that assumption \textup{\textbf{(H$_{c,\chi}$)}} holds true. Let $(\hat{t},x_{0})\in[0,T]\times\mathbb{R}^{n}$ be such that
\begin{equation}
\label{E:VU}
V(\hat{t},x_{0})\leqslant\mathcal{H}_{\sup}^{c}V(\hat{t},x_{0}), \qquad \qquad \qquad U(\hat{t},x_{0})<\mathcal{H}_{\inf}^{\chi}U(\hat{t},x_{0})
\end{equation}
or
\[
U(\hat{t},x_{0})\geqslant\mathcal{H}_{\inf}^{\chi}U(\hat{t},x_{0}).
\]
Then for every $\varepsilon>0$ there exists $\hat{x}\in\mathbb{R}^{n}$ such that
\[
V(\hat{t},x_{0})-U(\hat{t},x_{0})\leqslant V(\hat{t},\hat{x})-U(\hat{t},\hat{x})+\varepsilon
\]
and
\[
V(\hat{t},\hat{x})>\mathcal{H}_{\sup}^{c}V(\hat{t},\hat{x}), \qquad \qquad \qquad U(\hat{t},\hat{x})<\mathcal{H}_{\inf}^{\chi}U(\hat{t},\hat{x}).
\]
\end{lemma}
\begin{proof}
Fix $\varepsilon>0$. We divide the proof into two steps.\\
\emph{Step 1.} Without loss of generality, we can suppose that (\ref{E:VU}) holds at $(\hat{t},x_{0})$. As a matter of fact, suppose that $U(\hat{t},x_{0})\geqslant\mathcal{H}_{\inf}^{\chi}U(\hat{t},x_{0})$. Let $\alpha\in(0,1)$ to be chosen later, then there exists $z_{0}\in\mathscr{V}$ for which
\[
\mathcal{H}_{\inf}^{\chi}U(\hat{t},x_{0})\geqslant U(\hat{t},x_{0}+z_{0})+\chi(\hat{t},z_{0})-\alpha\varepsilon.
\]
Since $V$ is a subsolution, it satisfies $V(\hat{t},x_{0})\leqslant\mathcal{H}_{\inf}^{\chi}V(\hat{t},x_{0})$, therefore
\[
V(\hat{t},x_{0})-U(\hat{t},x_{0})\leqslant V(\hat{t},x_{0}+z_{0})-U(\hat{t},x_{0}+z_{0})+\alpha\varepsilon.
\]
Using condition (\ref{E:chi}), taking $\alpha$ sufficiently small, we can show that at $(\hat{t},x_{0}+z_{0})$ we have
\[
U(\hat{t},x_{0}+z_{0})<\mathcal{H}_{\inf}^{\chi}U(\hat{t},x_{0}+z_{0}).
\]
If $V(\hat{t},x_{0}+z_{0})>\mathcal{H}_{\sup}^{c}V(\hat{t},x_{0}+z_{0})$, we take $(\hat{t},\hat{x}):=(\hat{t},x_{0}+z_{0})$. Otherwise at $(\hat{t},x_{0}+z_{0})$ condition (\ref{E:VU}) holds true. \\
\emph{Step 2.} Suppose that (\ref{E:VU}) holds at $(\hat{t},x_{0})$. Let $\beta\in(0,1)$ to be chosen later, then there exists $y_{0}\in\mathscr{U}$ such that
\[
\mathcal{H}_{\sup}^{c}V(\hat{t},x_{0})\leqslant V(\hat{t},x_{0}+y_{0})-c(\hat{t},y_{0})+\beta\varepsilon.
\]
Since $U$ is a supersolution, it satisfies $U(\hat{t},x_{0})\geqslant\mathcal{H}_{\sup}^{c}U(\hat{t},x_{0})$, therefore
\[
V(\hat{t},x_{0})-U(\hat{t},x_{0})\leqslant V(\hat{t},x_{0}+y_{0})-U(\hat{t},x_{0}+y_{0})+\beta\varepsilon.
\]
Using condition (\ref{E:c}), taking $\beta$ sufficiently small, we can show that at $(\hat{t},x_{0}+y_{0})$ we have
\[
V(\hat{t},x_{0}+y_{0})>\mathcal{H}_{\sup}^{c}V(\hat{t},x_{0}+y_{0}).
\]
If $U(\hat{t},x_{0}+y_{0})<\mathcal{H}_{\inf}^{\chi}U(\hat{t},x_{0}+y_{0})$, we take $(\hat{t},\hat{x}):=(\hat{t},x_{0}+y_{0})$. Otherwise we can proceed as in \emph{Step 1} and we find $z_{0}\in\mathscr{V}$ for which
\[
\mathcal{H}_{\inf}^{\chi}U(\hat{t},x_{0}+y_{0})\geqslant U(\hat{t},x_{0}+y_{0}+z_{0})+\chi(\hat{t},z_{0})-\beta\varepsilon,
\]
\[
V(\hat{t},x_{0})-U(\hat{t},x_{0})\leqslant V(\hat{t},x_{0}+y_{0}+z_{0})-U(\hat{t},x_{0}+y_{0}+z_{0})+2\beta\varepsilon
\]
and
\[
U(\hat{t},x_{0}+y_{0}+z_{0})<\mathcal{H}_{\inf}^{\chi}U(\hat{t},x_{0}+y_{0}+z_{0}).
\]
Then, using condition (\ref{E:c}), we find (possibly reducing $\beta$)
\[
V(\hat{t},x_{0}+y_{0}+z_{0})>\mathcal{H}_{\sup}^{c}V(\hat{t},x_{0}+y_{0}+z_{0}).
\]
Therefore we define $(\hat{t},\hat{x}):=(\hat{t},x_{0}+y_{0}+z_{0})$.
\end{proof}
\begin{lemma}
\label{L:Uniq2}
Let $V,U\colon[0,T]\times\mathbb{R}^{n}\rightarrow\mathbb{R}$ be uniformly continuous on $[0,T)\times\mathbb{R}^{n}$ and suppose that assumption \textup{\textbf{(H$_{c,\chi}$)}} holds true. If $(\hat{t},\hat{x})\in[0,T)\times\mathbb{R}^{n}$ is such that
\[
V(\hat{t},\hat{x})>\mathcal{H}_{\sup}^{c}V(\hat{t},\hat{x}) \qquad \text{and} \qquad U(\hat{t},\hat{x})<\mathcal{H}_{\inf}^{\chi}U(\hat{t},\hat{x}),
\]
then there exists $\delta>0$ for which
\[
V(t,x)>\mathcal{H}_{\sup}^{c}V(t,x) \qquad \text{and} \qquad U(t,x)<\mathcal{H}_{\inf}^{\chi}U(t,x),
\]
when $(t,x)\in[(\hat{t}-\delta)\vee0,\hat{t}+\delta]\times\bar{B}_{\delta}(\hat{x})$, with $\hat{t}+\delta<T$.
\end{lemma}
\begin{proof}
Since $V(\hat{t},\hat{x})>\mathcal{H}_{\sup}^{c}V(\hat{t},\hat{x})$, there exists $\lambda>0$ such that
\[
V(\hat{t},\hat{x})>V(\hat{t},\hat{x}+y)-c(\hat{t},y)+\lambda, \qquad \forall\,y\in\mathscr{U}.
\]
Let $\varepsilon>0$ to be fixed later, then from the uniform continuity of $V$ there exists $\delta_{V}\in(0,\varepsilon)$ such that
\[
|V(t,x)-V(t',x')|\leqslant\varepsilon,
\]
when $|t-t'|\leqslant\delta_{V}$ and $|x-x'|\leqslant\delta_{V}$. Let $(t,x)\in[(\hat{t}-\delta_{V})\vee0,(\hat{t}+\delta_{V})\wedge T)\times\bar{B}_{\delta_{V}}(\hat{x})$ and $y\in\mathscr{U}$, then
\begin{align*}
V(\hat{t},\hat{x})-V(t,x)+V(t,x)>&\,V(\hat{t},\hat{x}+y)-V(t,x+y)+V(t,x+y) \\
&-c(\hat{t},y)+c(t,y)-c(t,y)+\lambda.
\end{align*}
Using the $1/2$-H\"older continuity of $c$ with respect to time, we find that there exists a constant $C>0$ such that
\[
V(t,x)>V(t,x+y)-c(t,y)-2\varepsilon-C\sqrt{\varepsilon}+\lambda.
\]
Therefore, taking $\varepsilon$ sufficiently small, we have
\[
V(t,x)>\mathcal{H}_{\sup}^{c}V(t,x), \qquad (t,x)\in[(\hat{t}-\delta_{V})\vee0,(\hat{t}+\delta_{V})\wedge T)\times\bar{B}_{\delta_{V}}(\hat{x}).
\]
Analogously, we can prove that there exists $\delta_{U}>0$ such that
\[
U(t,x)<\mathcal{H}_{\inf}^{\chi}U(t,x), \qquad (t,x)\in[(\hat{t}-\delta_{U})\vee0,(\hat{t}+\delta_{U})\wedge T)\times\bar{B}_{\delta_{U}}(\hat{x}).
\]
Taking $\delta=\min\{\delta_{V},\delta_{U},(T-\hat{t})/2\}$ we deduce the thesis.
\end{proof}

We are now in a position to prove the Comparison Theorem.
\begin{theorem}[Comparison Theorem]
\label{T:Uniq}
Let $V$ and $U$ be a viscosity subsolution and a viscosity supersolution to the HJBI equation (\ref{E:HJBI}), respectively. Suppose that assumptions \textup{\textbf{(H$_{b,\sigma}$)}}, \textup{\textbf{(H$_{f,g}$)}} and \textup{\textbf{(H$_{c,\chi}$)}} hold true and that the functions $V$ and $U$ are uniformly continuous on $[0,T)\times\mathbb{R}^{n}$. Then $V\leqslant U$ on $[0,T]\times\mathbb{R}^{n}$.
\end{theorem}
\begin{proof}
We argue by contradiction, assuming that
\[
\sup_{[0,T]\times\mathbb{R}^{n}}(V-U)>0.
\]
\emph{Step 1.} Let $\rho>0$ and introduce the functions
\[
\tilde{V}(t,x)=\text{e}^{\rho t}V(t,x) \qquad \text{ and } \qquad \tilde{U}(t,x)=\text{e}^{\rho t}U(t,x),
\]
for all $(t,x)\in[0,T]\times\mathbb{R}^{n}$. Then $\tilde{V}$ (resp., $\tilde{U}$) is a viscosity subsolution (resp., supersolution) to the following equation:
\begin{equation}
\label{E:HJBIrho}
\begin{cases}
\vspace{.2cm}
\max\Big\{\min\Big[\rho W-\dfrac{\partial W}{\partial t}-\mathcal{L}W-\tilde{f},W-\tilde{\mathcal{H}}_{\sup}^{c}W\Big], \\
\hspace{5cm}W-\tilde{\mathcal{H}}_{\inf}^{\chi}W\Big\}=0, \qquad &\textup{on }[0,T)\times\mathbb{R}^{n}, \\
W(T,x)=\tilde{g}(x), &\forall\,x\in\mathbb{R}^{n},
\end{cases}
\end{equation}
where, for every $(t,x)\in[0,T]\times\mathbb{R}^{n}$, $\tilde{f}(t,x)=\text{e}^{\rho t}f(x)$, $\tilde{g}(x)=\text{e}^{\rho T}g(x)$,
\[
\tilde{\mathcal{H}}_{\sup}^{c}W(t,x)=\sup_{y\in\mathscr{U}}\{W(t,x+y)-\text{e}^{\rho t}c(t,y)\}
\]
and
\[
\tilde{\mathcal{H}}_{\inf}^{\chi}W(t,x)=\inf_{z\in\mathscr{V}}\{W(t,x+z)+\text{e}^{\rho t}\chi(t,z)\}.
\]
\emph{Step 2.} Suppose that there exists $x_{0}\in\mathbb{R}^{n}$ such that $\tilde{V}(T,x_{0})-\tilde{U}(T,x_{0})>0$. Using Lemma \ref{L:Uniq} (we apply this lemma to $V$ and $U$, expressing the results in terms of $\tilde{V}$ and $\tilde{U}$), we derive the existence of $\hat{x}\in\mathbb{R}^{n}$ such that
\begin{equation}
\label{E:SupT}
\tilde{V}(T,\hat{x})-\tilde{U}(T,\hat{x})>0
\end{equation}
and
\[
\tilde{V}(T,\hat{x})>\mathcal{H}_{\sup}^{c}\tilde{V}(T,\hat{x}), \qquad \tilde{U}(T,\hat{x})<\mathcal{H}_{\inf}^{\chi}\tilde{U}(T,\hat{x}).
\]
From the subsolution property of $\tilde{V}$, we find $\tilde{V}(T,\hat{x})\leqslant\tilde{g}(\hat{x})$. Analogously, using the supersolution property of $\tilde{U}$, we have $\tilde{U}(T,\hat{x})\geqslant\tilde{g}(\hat{x})$. Therefore $\tilde{V}(T,\hat{x})-\tilde{U}(T,\hat{x})\leqslant0$, a contradiction with (\ref{E:SupT}).\\
\emph{Step 3.} Suppose now that there exists $(\bar{t},\bar{x})\in[0,T)\times\mathbb{R}^{n}$ such that $\tilde{V}(\bar{t},\bar{x})-\tilde{U}(\bar{t},\bar{x})>0$. Then, from Lemma \ref{L:Uniq} and Lemma \ref{L:Uniq2}, we deduce the existence of $(\hat{t},\hat{x})\in[0,T)\times\mathbb{R}^{n}$ and $\delta>0$ such that
\[
\sup_{I\times\bar{B}_{\delta}(\hat{x})}(\tilde{V}-\tilde{U})\geqslant\tilde{V}(\hat{t},\hat{x})-\tilde{U}(\hat{t},\hat{x})>0
\]
and
\[
\tilde{V}(t,x)>\mathcal{H}_{\sup}^{c}\tilde{V}(t,x), \qquad \tilde{U}(t,x)<\mathcal{H}_{\inf}^{\chi}\tilde{U}(t,x),
\]
for all $(t,x)\in I\times\bar{B}_{\delta}(\hat{x})$, where $I:=[\hat{t}-\delta,\hat{t}+\delta]\subset[0,T)$. We can also assume, without loss of generality, that
\[
\tilde{V}(t,x)-\tilde{U}(t,x)\leqslant0, \qquad \forall(t,x)\in I\times\partial\bar{B}_{\delta}(\hat{x}).
\]
Indeed, if this is not the case, define
\[
\hat{V}(t,x)=\tilde{V}(t,x)-\frac{M}{15}\Big(\frac{16|x-\hat{x}|^{4}}{\delta^{4}}\mathbbm{1}_{\{|x-\hat{x}|>\delta/2\}}(x)-1\Big), \qquad \forall(t,x)\in[0,T)\times\mathbb{R}^{n},
\]
where
\[
M:=\sup_{I\times\bar{B}_{\delta}(\hat{x})}(\tilde{V}-\tilde{U}).
\]
Since $J^{2,+}\hat{V}(t,x)=J^{2,+}\tilde{V}(t,x)$, for all $(t,x)\in[0,T)\times\mathbb{R}^{n}$, we simply replace $\tilde{V}$ with~$\hat{V}$.\\
\emph{Step 4.} Let $(t_{0},x_{0})\in I\times B_{\delta}(\hat{x})$ be such that
\begin{equation}
\label{E:Sup}
\sup_{I\times\bar{B}_{\delta}(\hat{x})}(\tilde{V}-\tilde{U})=(\tilde{V}-\tilde{U})(t_{0},x_{0})>0.
\end{equation}
For every integer $n\geqslant1$, consider the following function
\[
\Theta_{n}(t,x,y)=\tilde{V}(t,x)-\tilde{U}(t,y)-\varphi_{n}(t,x,y),
\]
with
\[
\varphi_{n}(t,x,y)=n|x-y|^{2}+|x-x_{0}|^{4}+|t-t_{0}|^{2},
\]
for every $(t,x,y)\in[0,T]\times\mathbb{R}^{n}\times\mathbb{R}^{n}$. For all $n$ there exists $(t_{n},x_{n},y_{n})\in I\times\bar{B}_{\delta}(\hat{x})\times\bar{B}_{\delta}(\hat{x})$ attaining the maximum of $\Theta_{n}$ on $I\times\bar{B}_{\delta}(\hat{x})\times\bar{B}_{\delta}(\hat{x})$. Then we have, up to a subsequence, $(t_{n},x_{n},y_{n})\in I\times B_{\delta}(\hat{x})\times B_{\delta}(\hat{x})$ and, when $n\rightarrow\infty$,
\begin{enumerate}[\upshape (i)]
    \item $(t_{n},x_{n},y_{n})\rightarrow(t_{0},x_{0},x_{0})$;
    \item $n|x_{n}-y_{n}|^{2}\rightarrow0$;
    \item $\tilde{V}(t_{n},x_{n})-\tilde{U}(t_{n},y_{n})\rightarrow\tilde{V}(t_{0},x_{0})-\tilde{U}(t_{0},x_{0})$.
\end{enumerate}
Indeed, up to a subsequence, $(t_{n},x_{n},y_{n})\rightarrow(\bar{t},\bar{x},\bar{y})\in I\times\bar{B}_{\delta}(\hat{x})\times\bar{B}_{\delta}(\hat{x})$. Observe that, for all $n$, the following holds true:
\[
\tilde{V}(t_{0},x_{0})-\tilde{U}(t_{0},x_{0})=\Theta_{n}(t_{0},x_{0},x_{0})\leqslant\Theta_{n}(t_{n},x_{n},y_{n}).
\]
Then we find
\begin{align}
\label{E:Theta_n}
\tilde{V}(t_{0},x_{0})-\tilde{U}(t_{0},x_{0})
&\leqslant\liminf_{n\rightarrow\infty}\Theta_{n}(t_{n},x_{n},y_{n})\leqslant\limsup_{n\rightarrow\infty}\Theta_{n}(t_{n},x_{n},y_{n})\leqslant \\
&\leqslant\tilde{V}(\bar{t},\bar{x})-\tilde{U}(\bar{t},\bar{y})-\liminf_{n\rightarrow\infty}n|x_{n}-y_{n}|^{2}-|\bar{x}-x_{0}|^{4}-|\bar{t}-t_{0}|^{2}. \notag
\end{align}
As a consequence, up to a subsequence, $\lim_{n\rightarrow\infty}n|x_{n}-y_{n}|^{2}<\infty$, from which we deduce $\bar{x}=\bar{y}$. Again from (\ref{E:Theta_n}), using the optimality of $(t_{0},x_{0})$, we derive (i) and (ii). Consequently we get also (iii). Finally, since $x_{0}\in B_{\delta}(\hat{x})$, up to a subsequence, we deduce that $(t_{n},x_{n},y_{n})\in I\times B_{\delta}(\hat{x})\times B_{\delta}(\hat{x})$.\\
\emph{Step 5.} We may apply Ishii's lemma (see Theorem 8.3 in \cite{CIL92}) to the sequence $\{(t_{n},x_{n},y_{n})\}_{n}$: there exist $(p_{\tilde{V}}^{n},q_{\tilde{V}}^{n},M_{n})\in\bar{J}^{2,+}\tilde{V}(t_{n},x_{n})$ and $(p_{\tilde{U}}^{n},q_{\tilde{U}}^{n},N_{n})\in\bar{J}^{2,-}\tilde{U}(t_{n},y_{n})$ such that
\[
p_{\tilde{V}}^{n}-p_{\tilde{U}}^{n}=\frac{\partial\varphi_{n}}{\partial t}(t_{n},x_{n},y_{n})=2(t_{n}-t_{0}),
\]
\[
q_{\tilde{V}}^{n}=D_{x}\varphi_{n}(t_{n},x_{n},y_{n}), \qquad q_{\tilde{U}}^{n}=-D_{y}\varphi_{n}(t_{n},x_{n},y_{n})
\]
and
\[
\bigg(
\begin{array}{cc}
M_{n} & 0 \\
0 & - N_{n}
\end{array}
\bigg)\leqslant A_{n}+\frac{1}{2n}A_{n}^{2},
\]
where $A_{n}=D_{xy}^{2}\varphi_{n}(t_{n},x_{n},y_{n})$. From the viscosity subsolution property of $\tilde{V}$ we find
\[
\rho\tilde{V}(t_{n},x_{n})-p_{\tilde{V}}^{n}-\langle b(x_{n}),q_{\tilde{V}}^{n}\rangle-\tfrac{1}{2}\text{tr}[(\sigma\sigma')(x_{n})M_{n}]-f(x_{n})\leqslant0.
\]
Analogously, from the viscosity supersolution property of $\tilde{U}$ we have
\[
\rho\tilde{U}(t_{n},y_{n})-p_{\tilde{U}}^{n}-\langle b(y_{n}),q_{\tilde{U}}^{n}\rangle-\tfrac{1}{2}\text{tr}[(\sigma\sigma')(y_{n})N_{n}]-f(y_{n})\geqslant0.
\]
By subtracting the two previous inequalities, we obtain
\begin{align}
\label{E:Ishii}
\rho(\tilde{V}(t_{n},x_{n})-\tilde{U}(t_{n},y_{n}))\leqslant&\,p_{\tilde{V}}^{n}-p_{\tilde{U}}^{n}
+\langle b(x_{n}),q_{\tilde{V}}^{n}\rangle-\langle b(y_{n}),q_{\tilde{U}}^{n}\rangle+ \\
&+\tfrac{1}{2}\text{tr}[(\sigma\sigma')(x_{n})M_{n}]-\tfrac{1}{2}\text{tr}[(\sigma\sigma')(y_{n})N_{n}]+ \notag \\
&+f(x_{n})-f(y_{n}). \notag
\end{align}
When $n\rightarrow\infty$,
\[
p_{\tilde{V}}^{n}-p_{\tilde{U}}^{n}=2(t_{n}-t_{0})\rightarrow0,
\]
Moreover, from the Lipschitzianity of $b$ and (ii),
\[
\lim_{n\rightarrow\infty}\big(\langle b(x_{n}),q_{\tilde{V}}^{n}\rangle-\langle b(y_{n}),q_{\tilde{U}}^{n}\rangle\big)=0
\]
Finally, from the Lipschitzianity of $\sigma$, (i) and (ii), we get
\[
\limsup_{n\rightarrow\infty}\big(\tfrac{1}{2}\text{tr}[(\sigma\sigma')(x_{n})M_{n}]-\tfrac{1}{2}\text{tr}[(\sigma\sigma')(y_{n})N_{n}]\big)\leqslant0.
\]
Since, thanks to (iii), the left-hand side of inequality (\ref{E:Ishii}) goes to $\rho(\tilde{V}(t_{0},x_{0})-\tilde{U}(t_{0},y_{0}))$, we find
$\tilde{V}(t_{0},x_{0})-\tilde{U}(t_{0},y_{0})\leqslant0$, a contradiction with (\ref{E:Sup}).
\end{proof}

Thanks to the Comparison Theorem we deduce that the stochastic differential game admits a value, as stated in the following corollary.
\begin{corollary}
Under assumptions \textup{\textbf{(H$_{b,\sigma}$)}}, \textup{\textbf{(H$_{f,g}$)}} and \textup{\textbf{(H$_{c,\chi}$)}} the lower and upper value functions coincide and the value function of the stochastic differential game is given by $V(t,x):=V^{-}(t,x)=V^{+}(t,x)$, for every $(t,x)\in[0,T)\times\mathbb{R}^{n}$.
\end{corollary}
\begin{proof}
We know from Proposition \ref{P:Cont-x} and Proposition \ref{P:Cont-t} that the two value functions are uniformly continuous on $[0,T)\times\mathbb{R}^{n}$. Furthermore, thanks to Theorem~\ref{T:Exist}, both $V^{-}$ and $V^{+}$ are viscosity solutions to the HJBI equation~(\ref{E:HJBI}). Hence, from the Comparison Theorem, we deduce the thesis.
\end{proof}

\paragraph{Acknowledgements.}
The author would like to take this opportunity to thank Professor Marco Fuhrman and Professor Huy\^en Pham for helpful discussions and suggestions related to this work. The author would also like to thank Professor Martino Bardi, Professor Pierre Cardaliaguet and Professor Panagiotis E. Souganidis for useful remarks.

\end{document}